\documentclass[11pt, oneside]{article}   	
\usepackage{geometry}                		
\usepackage{graphicx}				
\usepackage{amssymb}
\usepackage{amsmath}
\usepackage{amsthm}
\usepackage{harpoon}
\usepackage{accents}

\usepackage{tikz}  
\usepackage{float}
\restylefloat{table}
\usepackage{mathrsfs}
\usepackage{verbatim}
\usepackage{enumitem}  

\usetikzlibrary{positioning,chains,fit,shapes,calc}  

\newtheorem{theorem}{Theorem}
\newtheorem{lemma}{Lemma}
\newtheorem{definition}{Definition}

\newtheorem{remark}{Remark}
\newtheorem{corollary}{Corollary}

\newcommand{\vect}[1]{\accentset{\rightharpoonup}{#1}}
\newcommand{\harpoon}{\overset{\rightharpoonup}}
\newcommand{\ppp}{\[\begin{aligned}}
\newcommand{\ooo}{\end{aligned}\]}

\begin{document}

\title{On the set of fixed points for NRS($m$)}
\author{Mario DeFranco}
\maketitle

\abstract{Let $f(z)$ be a degree $d$ polynomial with zeros $z_i$. For arbitrary $m$ we construct explicit set of fixed points (attractors) of NRS($m$), and prove a factored formula for the Jacobian at these points.  We prove that if NRS(2), when applied to $f$ with an arbitrary starting point, converges to a point $(w_0, w_1)$, then $w_0$ is of the form $z_i+z_j$ for some $i \neq j$. As a corollary, we prove a formula expressing the  elementary symmetric expansion of the function 
\[
\prod_{1\leq i < j \leq d} (z - z_i -z_j)
\] 
in the variables $z_i$ in terms of non-intersecting paths on certain directed graphs, using the Lindstr\"om-Gessel-Veinnot Lemma. }

\section{Introduction} 

Let $m\geq1$ be an integer. We prove some results about NRS($m$), the generalization of the Newton-Raphson-Simpson method defined in \cite{DeFranco 1}. We will use the same notation specified in that paper unless otherwise specified. 

Suppose $a_i \in \mathbb{C}$ and let $z_i$ denote the zeros of a degree $d$ polynomial $f(z)$:
\begin{align*}
f(z) =& \sum_{i=0}^d a_i z^i \\ 
=& a_0 \prod_{i=1}^d (1-\frac{z}{z_i})
\end{align*}
where we assume $a_0\neq 0$, $z_i \neq 0$ and  $a_m \neq 0$. Let $\mathrm{attractors}(m) \subset \mathbb{C}^m$ denote the set of attractors (fixed points) of NRS($m$).  
In Section \ref{subset}, we explicitly construct a subset $V(m) \subset \mathrm{attractors}(m)$. The subset $V(m)$ is indexed by subsets of size $m$ of 
\[
\{ z_1, \ldots, z_d\},
\]
so generically
\[
|V(m)| = {d \choose m}.
\]
Each point in $V(m)$ has coordinates that are Laurent poynomials in $z_i$. 
 
 In Section \ref{Jacobian value}, we explicitly construct a factored expression for the value of the Jacobian of $\harpoon{f}$ evaluated at each point in $V(m)$. It is well-known that the non-vanishing of this Jacobian implies quadratic convergence of the multi-dimensional Newton-Raphson-Simpson method. 
 
 Let $V(0,m)$ denote the set of the first coordinates of the points in $V(m)$, i.e. if 
 \[
 (\alpha_0, \ldots, \alpha_{m-1}) \in V(m) 
 \]
 then 
 \[
 \alpha_0 \in V(0,m).
 \]
 Similarly define $ \mathrm{attractors}(0,m)$.
In Section \ref{set eq}, we prove that 
\[
V(0,2) =  \mathrm{attractors}(0,2)
\]
by proving that there exists a degree-${d \choose 2}$ polynomial $g(x_0)$ in the ideal generated by $f_{0,2}(x_0, x_1) - x_0$ and $f_{1,2}(x_0, x_1)-x_1$. We express the coefficients in the expansion of $g(x_0)$ using non-intersecting (vertex-disjoint) paths in certain directed graphs using the Viennot-Lindstr\"om-Gessel Lemma. 

In Section \ref{related} we prove some more results about these graphs with vertex-disjoint paths. 

\section{$V(m) \subset \mathrm{attractors}(m)$} \label{subset}

\begin{definition}
For integer $k \geq 0$, 
let $e_k(t_1, \ldots, t_n)$ denote the elementary symmetric function 
\[
e_k(t_1, \ldots, t_r) = \sum_{1\leq i_1 < \ldots < i_k \leq r} \prod_{j=1}^k t_{i_j}
\] 
and  $h_k(t_1, \ldots, t_r)$ the complete homogeneous symmetric function
\[
h_k(t_1, \ldots, t_n) = \sum_{1\leq i_1 \leq \ldots \leq i_k \leq r} \prod_{j=1}^k t_{i_j}.
\] 
If $k >r$, then
\[
e_k(t_1, \ldots, t_r) =0. 
\]
Let $C(n,k)$ denote the set of compositions of $n$ into $k$ positive integer parts. 
For $c \in C(n,k)$, let $|c|$ denote $k$.
Let $e_c(t_1, \ldots, t_n)$ denote
\[
e_c(t_1, \ldots, t_n) = \prod_{i=1}^{|c|} e_{c(i)}(t_1, \ldots, t_n)
\]
and similarly for $h_c(t_1, \ldots, t_n)$. 
Define 
\[
\alpha_0 = \sum_{i=1}^m z_i
\]
and for $1 \leq l \leq m-1$
\[
\alpha_l = h_{l-1}(\frac{1}{z_{m+1}}, ..., \frac{1}{z_d}) (-\frac{a_{m-1}}{a_m})^l \frac{a_m}{a_0}(-1)^m\prod_{i=1}^m z_i . 
\]
 Permuting the $z_i$ in the expression for $\harpoon{\alpha}$ generically results in ${d \choose m}$ points. Call the resulting set of points $V(m)$. 
\end{definition}

We claim that 
\[
\alpha=(\alpha_0, \ldots, \alpha_{m-1}) \in \mathrm{attractors}(m)
\]
 i.e. 
\[
\vect{f}_m(\alpha) = \vect{0}
\]

\begin{remark}
We will use for $0 \leq i \leq d$
\[
(-1)^i\frac{a_i}{a_0} = e_i(\frac{1}{z_1}, \ldots, \frac{1}{z_d}).
\]
\end{remark}

\begin{lemma} \label{identities}
For integer $i \geq 0$ and composition $c$, we let $e_i$ and $e_c$ denote $e_i(t_1, \ldots, t_j)$ and $e_c(t_1, \ldots, t_n)$, and similarly for $h_i$ and $h_c$. 

\begin{equation} \label{e h}
\sum_{i=0}^n (-1)^i h_i e_{n-i} = 
\begin{cases}
1  \text{ if } n=0 \\ 
0  \text{ if } n > 0.
\end{cases}
\end{equation}

\begin{equation} \label{ec}
\sum_{c \in C(n)} (-1)^{|c|}e_{c}(\vect{t})  = (-1)^n h_n(\vect{t})
\end{equation}
and 
\begin{equation} \label{hc}
\sum_{c \in C(n)} (-1)^{|c|} h_c(\vect{t})  =(-1)^n e_n(\vect{t})
\end{equation}
\end{lemma}

\begin{proof}
We first prove identity \eqref{e h}. We fix $n$ and let 
\[
e_i = e_i(t_1, \ldots, t_n) \text{ and } h_i = h_i(t_1, \ldots, t_n).
\]
The generating function for $e_i$ is
\begin{align*}
\prod_{i=1}^n (1-x t_i) &= \sum_{i=1}^n (-1)^i x^i e_i \\ 
= \sum_{i=1}^\infty (-1)^i x^i e_i
\end{align*}
where we have used $e_i =0$ for $i >n$.
The reciprocal of the generating function for $e_i$ is the generating function for $h_i$
\[
(\prod_{i=1}^n (1-x t_i))^{-1} = \sum_{i=0}^\infty  x^i h_i
\] 
where we use
\[
(1-x t_i)^{-1} = \sum_{j=0}^\infty x^j t_i^j. 
\]
Multiplying these two generating functions yields 
\[
1  = \sum_{j=0}^\infty x^j \sum_{i=0}^j (-1)^i e_i h_{j-i}.
\]
Comparing the coefficients of $x^j$ proves \eqref{e h}. 

Identities \eqref{ec} and \eqref{hc} follow from \eqref{e h} by straightforward induction. This completes the proof.
\end{proof}

\begin{lemma} \label{e z prod}
For $n \leq m$:
\[
e_n(\frac{1}{z_{1}}, ..., \frac{1}{z_d}) \prod_{i=1}^m z_i =\sum_{i=0}^n e_{m-n+i}(z_1, ..., z_m)e_{i}(\frac{1}{z_{m+1}}, ..., \frac{1}{z_d})
\]
and for $n \geq m$:
\[
e_n(\frac{1}{z_{1}}, ..., \frac{1}{z_d}) \prod_{i=1}^m z_i =\sum_{i=0}^m e_{i}(z_1, ..., z_m)e_{n-m+i}(\frac{1}{z_{m+1}}, ..., \frac{1}{z_d})
\]
\end{lemma}
\begin{proof}
This follows from expanding the left sides and collecting terms. 
\end{proof} 

\begin{lemma} \label{block expr alpha}
For integer $k \geq 1$,
 \[
 \mathrm{expr}_m(\alpha;k) = (-1)^{k-1} e_{k}(z_1,...,z_m). 
 \]
\end{lemma}

\begin{proof} 
When $k=1$, we have 
\[
 \mathrm{expr}_m(x;1) = x_0.
\]
Evaluating at $\vect{\alpha}$ and applying the definition gives
\[
\alpha_0 = \sum_{i=1}^m z_i
\]
which is $e_1(z_1,...,z_m)$. 

For $k \geq 2$, we have by definition
\begin{align*}
 \mathrm{expr}_m(\vect{\alpha};k)  &=  \sum_{h=k-1}^{m-1}\mathrm{expr}_m(\vect{\alpha}; k,h)\\ 
 & =  \sum_{h=k-1}^{m-1} h_{h-(k-2)-1}(\frac{1}{z_{m+1}}, ..., \frac{1}{z_d} )(-1)^{m+1} \frac{a_{m-1-h}}{a_0} \prod_{i=1}^m z_i. \\
\end{align*}
We apply Lemma \ref{e z prod} and re-arrange this to obtain 
\begin{align*}
\sum_{j=0}^{m-k} (-1)^{k-1}e_{k+j}(z_1,...,z_m) \sum_{i=0}^j (-1)^i h_i(\frac{1}{z_{m+1}}, ..., \frac{1}{z_d} ) e_{j-i}(\frac{1}{z_{m+1}}, ..., \frac{1}{z_d} ).
\end{align*}
Applying identity \eqref{e h} yields 
\[
 (-1)^{k-1}e_{k}(z_1,...,z_m).
\]
This completes the proof. 
\end{proof}

\begin{lemma} \label{PT eval}
For $2\leq s \leq d-m+1$,
 \begin{align*}
\mathrm{PT}_m(\vect{\alpha};s) &= (-1)^{s-m} \frac{a_0 e_{s-1}(\frac{1}{z_{m+1}}, ..., \frac{1}{z_d} )}{a_m \prod_{i=1}^m z_i} \\
\mathrm{PT}_m(\vect{\alpha};1) &= 1- (-1)^m \frac{a_0}{a_m \prod_{i=1}^m z_i}\\
\mathrm{PT}_m(\vect{\alpha};0) &= \frac{a_{m-1}}{a_m} + \sum_{i=1}^m z_i 
 \end{align*} 
 
\end{lemma}
\begin{proof} 
For $0\leq s \leq d-m+1$, by definition

\begin{align}
\mathrm{PT}(\alpha,s)&=\sum_{i=0}^{d-m-s+1} -\frac{a_{m-1+i+s}}{a_m}\mathbf{1}(i+s \geq 2)\sum_{c \in C(i)}  \prod_{j=1}^{\mathrm{length}(c)} \mathrm{expr}(\alpha;c(j)) \nonumber\\ 
                     &= \sum_{i=0}^{d-m-s+1} -\frac{a_{m-1+i+s}}{a_m}\mathbf{1}(i+s \geq 2)\sum_{c \in C(i)}  \prod_{j=1}^{\mathrm{length}(c)}(-1)^{c(j)-1} e_{c(j)}(z_1,...,z_m) \label{ref block expr alpha} \\ 
                     &= \sum_{i=0}^{d-m-s+1} -\frac{a_{m-1+i+s}}{a_m}\mathbf{1}(i+s \geq 2) (-1)^i\sum_{c \in C(i)}  (-1)^{|c|} e_{c}(z_1,...,z_m) \nonumber\\  
                     &= \sum_{i=0}^{d-m-s+1} -\frac{a_{m-1+i+s}}{a_m}\mathbf{1}(i+s \geq 2)h_{i}(z_1,...,z_m)     \label{ref ec}
\end{align}
where at line \eqref{ref block expr alpha} we use Lemma \ref{block expr alpha} and at line \eqref{ref ec} we have used identity \eqref{ec}.

Multiplying through by $\displaystyle  \prod_{i=1}^m z_i $ and applying Lemma \ref{e z prod} gives
\begin{align*}
 &\mathrm{PT}_m(\vect{\alpha};s) \prod_{i=1}^m z_i \\ 
 &= \frac{-a_0}{a_m}\sum_{j=0}^{d-m-s+1} (-1)^{m-1+s+j} \mathbf{1}(s+j \geq 2) h_{j}(z_1, ..., z_m) \sum_{i=0}^m e_i(z_1, ..., z_m)e_{s-1+i+j}(\frac{1}{z_{m+1}}, ..., \frac{1}{z_d} ).   
\end{align*} 

Re-arranging gives 
\[
(-1)^{s-m}\frac{a_0}{a_m} \sum_{j=0}^{d-m-s+1} e_{s-1+j}(\frac{1}{z_{m+1}}, ..., \frac{1}{z_d}) \sum_{i=0}^j (-1)^i h_i(z_1, ..., z_m)e_{j-i}(z_1, ..., z_m) \mathbf{1}(s+i \geq 2). 
\]
When $s \geq 2$, we apply identity \eqref{e h} to see that the sum reduces to  
\[
(-1)^{s-m}\frac{a_0}{a_m} e_{s-1}(\frac{1}{z_{m+1}}, ..., \frac{1}{z_d}). 
\]

When $s=1$, we have from line \eqref{ref ec}
\begin{align*}
 \mathrm{PT}_m(\vect{\alpha};1) &= \sum_{j=0}^{d-m-1+1}  (-\frac{a_{m-1+1+j}}{a_m}) h_{j}(z_1, ..., z_m)\mathbf{1}(1+j \geq 2)\\ 
 & =  -(-\frac{a_{m-1+1+0}}{a_m}) h_{0}(z_1, ..., z_m)+\sum_{j=0}^{d-m-1+1}  (-\frac{a_{m-1+1+j}}{a_m}) h_{j}(z_1, ..., z_m)\\ 
\end{align*}
which from the same reasoning above is equal to 
\[
1 - \frac{a_0}{(-1)^m a_m \prod_{i=1}^m z_i}.
\] 

Now for $s=0$, we have from above
\begin{align*}
 &\mathrm{PT}_m(\vect{\alpha};0) \prod_{i=1}^m z_i \\ 
 &= \frac{-1}{a_m}\sum_{j=2}^{d-m+1} (-1)^{m-1+j} h_{j}(z_1, ..., z_m) \sum_{i=0}^m e_i(z_1, ..., z_m)e_{-1+i+j}(\frac{1}{z_{m+1}}, ..., \frac{1}{z_d} ).   
\end{align*} 
Adding and subtracting the terms 
\[
(\frac{-1}{a_m}\prod_{i=1}^m z_i)(-\frac{a_{m-1+1}}{a_m})h_1(z_1,...,z_m) +\frac{(-1)^m}{a_m}(\prod_{i=1}^m z_i)\sum_{i=0}^{m-1} e_{i+1}(z_1, ..., z_m)e_i(\frac{1}{z_{m+1}}, ..., \frac{1}{z_d} )
\]
and re-arranging yields
\begin{align*}
&\frac{(-1)^{m}}{a_m} \sum_{j=1}^{d-m+1} e_{-1+j}(\frac{1}{z_{m+1}}, ..., \frac{1}{z_d}) \sum_{i=0}^j (-1)^i h_i(z_1, ..., z_m)e_{j-i}(z_1, ..., z_m) \mathbf{1}(s+i \geq 2)\\ 
& -(-\frac{a_{m-1+1}}{a_m})h_1(z_1,...,z_m)(\prod_{i=1}^m z_i)  - (-1)^{m} \frac{(-1)^{m-1}a_{m-1}}{a_m} (\prod_{i=1}^m z_i).
\end{align*}
Applying identity \eqref{e h} and simplifying yield
\[
(\prod_{i=1}^m z_i)( \frac{a_{m-1}}{a_m}+\sum_{i=1}^m z_i ) .
\]
This completes the proof.
\end{proof} 

\begin{theorem} \label{t alpha}
\[
\vect{f}_m(\vect{\alpha}) =  (0,\ldots, 0).
\]
\end{theorem} 
\begin{proof}
We prove for $0 \leq i \leq m-1$
\[
\alpha_i = f_{i,m}(\alpha).
\]
Using Lemma \ref{PT eval} and the definition of $f_{i,m}(x)$ we check 
\[
\alpha_0 = -\frac{a_{m-1}}{a_m} +\mathrm{PT}_m(\vect{\alpha};0)  
\]
and 
\[
\alpha_1 =  \frac{-a_{m-1}}{a_m}+\alpha_1\mathrm{PT}_m(\vect{\alpha};1).  
\]
For $i \geq 2$ we must prove 
\[
\alpha_i = \sum_{j=0}^{i-1} \alpha_{i-j} (\frac{-a_{m-1}}{a_m})^j \mathrm{PT}_m(\vect{\alpha};j+1).
\]
 Using Lemma \ref{PT eval} on the right side gives 
 \begin{align*}
&(-1)^m \frac{a_m}{a_0}(\frac{-a_{m-1}}{a_m})^i h_{i-1}(\frac{1}{z_{m+1}}, ..., \frac{1}{z_d})(\prod_{i=1}^m z_i)(1-\frac{a_0}{(-1)^m a_m \prod_{i=1}^m z_i }) \\
&+(\frac{-a_{m-1}}{a_m})^i
 \sum_{j=1}^{i-1}  (-1)^{j+1}e_{j}(\frac{1}{z_{m+1}}, ..., \frac{1}{z_d}) h_{i-1-j}(\frac{1}{z_{m+1}}, ..., \frac{1}{z_d}))  \\ 
 &= \alpha_i + (\frac{-a_{m-1}}{a_m})^i
 \sum_{j=0}^{i-1}  (-1)^{j+1}e_{j}(\frac{1}{z_{m+1}}, ..., \frac{1}{z_d}) h_{i-1-j}(\frac{1}{z_{m+1}}, ..., \frac{1}{z_d}))\\ 
 & = \alpha_i. 
 \end{align*}
 This completes the proof.
\end{proof}

\section{Formula for $\det(J_{{\harpoon f}_m}(\harpoon \alpha))$  
} \label{Jacobian value}

\begin{lemma}For integer $1\leq i \leq m-1$,
\begin{align*}
&\frac{\partial \mathrm{PT}(x,s)}{\partial x_i} = \sum_{j=0}^{d-m-s+1} \frac{-a_{m-1+j+s}}{a_m}\mathbf{1}(j+s \geq 2) \\ 
&\times \sum_{k=2}^{m-i+1} (\frac{-a_{m-1}}{a_m})^{-i} (\frac{-a_{m-1-i-(k-2)}}{a_m}) \sum_{l=0}^{j-k} \left(\sum_{c \in C(l)} \prod_{n=1}^{|c|} \mathrm{expr}_m(x; c(n)) \right) \left(\sum_{c \in C(j-k-l)} \prod_{n=1}^{|c|} \mathrm{expr}_m(x; c(n)) \right) \\ 
\end{align*}
and
\begin{align*}
&\frac{\partial \mathrm{PT}(x,s)}{\partial x_0} = \sum_{j=0}^{d-m-s+1} \frac{-a_{m-1+j+s}}{a_m}\mathbf{1}(j+s \geq 2) \\ 
&\times \sum_{l=0}^{j-1} \left(\sum_{c \in C(l)} \prod_{n=1}^{|c|} \mathrm{expr}_m(x; c(n)) \right) \left(\sum_{c \in C(j-1-l)} \prod_{n=1}^{|c|} \mathrm{expr}_m(x; c(n)) \right).
\end{align*} 
\end{lemma}

\begin{proof}
From the definition of $\mathrm{PT}(x,s)$ we differentiate using the product rule
\[
\frac{\partial }{\partial x_i} \sum_{c \in C(j)} \prod_{n=1}^{|c|} \mathrm{expr}_m(x; c(n)).
\]
Now $x_i$ for $i \geq 1$ appears in $\mathrm{expr}_m(x; k)$ if $2 \leq k \leq m-i+1$. For such a $k $ suppose there is a $c \in C(j)$ with $c(r)=k$ for some $r$. Then let $l$ be the integer such that 
\[
l = \sum_{n=1}^{r-1} c(n) \text{ and } j-k-l = \sum_{n=r+1}^{|c|} c(r).
\] 
Summing over all these possibilities gives 
\ppp
&\frac{\partial }{\partial x_i} \sum_{c \in C(j)} \prod_{n=1}^{|c|} \mathrm{expr}_m(x; c(n)) \\ 
&=\sum_{k=2}^{m-i+1} \frac{\partial}{\partial x_i}  \mathrm{expr}_m(x; k)\sum_{l=0}^{j-k} \left(\sum_{c \in C(l)} \prod_{n=1}^{|c|} \mathrm{expr}_m(x; c(n)) \right) \left(\sum_{c \in C(i-k-l)} \prod_{n=1}^{|c|} \mathrm{expr}_m(x; c(n)) \right). 
\ooo
Therefore
\begin{align*}
\frac{\partial \mathrm{PT}(x,s)}{\partial x_i} &= \sum_{j=0}^{d-m-s+1} \frac{-a_{m-1+j+s}}{a_m}\mathbf{1}(j+s \geq 2) \\ 
&\times \sum_{k=2}^{m-i+1} \frac{\partial}{\partial x_i}  \mathrm{expr}_m(x; k)\sum_{l=0}^{j-k} \left(\sum_{c \in C(l)} \prod_{n=1}^{|c|} \mathrm{expr}_m(x; c(n)) \right) \left(\sum_{c \in C(j-k-l)} \prod_{n=1}^{|c|} \mathrm{expr}_m(x; c(n)) \right). 
\end{align*}
and from the definition we have for $i \geq 1$
\[
\frac{\partial}{\partial x_i}  \mathrm{expr}_m(x; k) = (\frac{-a_{m-1}}{a_m})^{-i} (\frac{-a_{m-1-i-(k-2)}}{a_m}).
\]
This completes the proof for $i \geq 1$. 

For $i=0$, $x_0$ only appears in $\mathrm{expr}_m(x; 1)$. By the same reasoning we have 
\begin{align*}
\frac{\partial \mathrm{PT}(x,s)}{\partial x_0} &= \sum_{j=0}^{d-m-s+1} \frac{-a_{m-1+j+s}}{a_m}\mathbf{1}(j+s \geq 2) \\ 
&\times \sum_{l=0}^{j-1} \left(\sum_{c \in C(l)} \prod_{n=1}^{|c|} \mathrm{expr}_m(x; c(n)) \right) \left(\sum_{c \in C(j-1-l)} \prod_{n=1}^{|c|} \mathrm{expr}_m(x; c(n)) \right). 
\end{align*}
This completes the proof.
\end{proof}
\begin{corollary}
For $1 \leq i \leq m-1$
\ppp
&\frac{\partial \mathrm{PT}(\vect{x},s)}{\partial x_i} |_{\vect{x} = \vect{\alpha}}= \sum_{j=0}^{d-m-s+1} \frac{-a_{m-1+j+s}}{a_m}\mathbf{1}(j+s \geq 2) \\ 
&\times \sum_{k=2}^{m-i+1} (\frac{-a_{m-1}}{a_m})^{-i} (\frac{-a_{m-1-i-(k-2)}}{a_m}) \sum_{l=0}^{j-k}  h_l(z_1, ..., z_m) h_{j-k-l}(z_1, ..., z_m)\\ 
\ooo
and 
\ppp
&\frac{\partial \mathrm{PT}(\vect{x},s)}{\partial x_0} |_{\vect{x} = \vect{\alpha}} = \sum_{j=0}^{d-m-s+1} \frac{-a_{m-1+j+s}}{a_m}\mathbf{1}(j+s \geq 2) \\ 
&\times \sum_{l=0}^{j-1}  h_l(z_1, ..., z_m)  h_{j-1-l}(z_1, ..., z_m) .
\ooo
\end{corollary}
\begin{proof}
This follows from Lemma  \ref{block expr alpha} and identities \eqref{ec} and \eqref{hc}. 
\end{proof}

\begin{lemma} \label{hh to he}
For a non-negative integer $j \leq d$, 
\[
\sum_{i=0}^{d-j} a_{j+i} \sum_{k=0}^i h_{k}(z_1, ..., z_m)h_{i-k}(z_1, ..., z_m)  = \frac{(-1)^d}{\prod_{i=1}^d z_i} \sum_{k=0}^{d-j}(-1)^kh_{d-j-k}(z_1,...,z_m)e_k(z_{m+1}, ...,z_d ).
\]
\end{lemma}
\begin{proof}
We multiply through by 
\[
(-1)^d \prod_{i=1}^d z_i
\]
and use 
\[
a_{j+i}\prod_{i=1}^d z_i = (-1)^{j+i}\sum_{g=0}^{d-j-i}e_{d-j-i-g}(z_1, ..., z_m)e_g(z_{m+1}, ..., z_d).
\]
We now set $l=d-j$ and see that the lemma statement is equivalent to 
\begin{align*}
&\sum_{i=0}^{l} (-1)^{l-i} (\sum_{g=0}^{l-i} e_{l-i-g}(z_1, ..., z_m) e_{g}(z_{m+1}, ..., z_d) ) (\sum_{k=0}^i h_{k}(z_1, ..., z_m)h_{i-k}(z_1, ..., z_m) ) \\
&= \sum_{g=0}^{l}(-1)^g h_{l-g}(z_1,...,z_m)e_g(z_{m+1}, ...,z_d )
\end{align*}
We expand the left side and take a term indexed by the 3-tuple $(i,g,k)$  
\[
(-1)^{l-i}e_{l-i-g}(z_1, ..., z_m) e_{g}(z_{m+1}, ..., z_d) h_{k}(z_1, ..., z_m)h_{i-k}(z_1, ..., z_m). 
\]
For fixed $(i,g,k)$ we group together all the terms of the form 
\[
(i+n, g, k+n)
\]
for integer $n$. The sum of these terms is equal to 0 by identity \ref{e h} unless $k=0$ and $l-i-g=0$, in which case we only have one term in the group which is 
\[
(-1)^g h_{l-g}(z_1, ..., z_m)e_{g}(z_{m+1}, ..., z_d).
\]
The index $g$ ranges from 0 to $l$. This completes the proof. 
\end{proof}

\begin{lemma}For $ 1 \leq i\leq m-1$
 \ppp
&\frac{\partial \mathrm{PT}(\vect{x},s)}{\partial x_i} |_{\vect{x}=\vect{\alpha}} = (\frac{-a_{m-1}}{a_m})^{-i}  \frac{(-1)^d}{-a_m\prod_{n=1}^d z_n}\\ 
&\times \sum_{k=2}^{m-i+1} (\frac{-a_{m-1-i-(k-2)}}{a_m})\sum_{l=0}^{d-(m-1+k+s)} (-1)^l h_{d - (m-1+k+s)-l}(z_1, \ldots, z_m) e_l(z_{m+1}, \ldots, z_d)
\ooo

 \ppp
&\frac{\partial \mathrm{PT}(\vect{x},s)}{\partial x_0} |_{\vect{x} = \vect{\alpha}} = \frac{(-1)^d}{-a_m\prod_{n=1}^d z_n} \sum_{l=0}^{d-(m+s)} (-1)^l h_{d - (m+s)-l}(z_1, \ldots, z_m) e_l(z_{m+1}, \ldots, z_d)+\mathbf{1}(s=0)
\ooo
\end{lemma}
\begin{proof}
From Corollary 1 \eqref{} we re-arrange to obtain 
\ppp
&\frac{\partial \mathrm{PT}(\vect{x},s)}{\partial x_i} |_{\vect{x} = \vect{\alpha}}\\ 
& =\sum_{k=2}^{m-i+1} (\frac{-a_{m-1-i-(k-2)}}{a_m})\sum_{j=k}^{d-m-s+1}\frac{-a_{m-1+j+s}}{a_m} \mathbf{1}(j+s \geq 2) \sum_{l=0}^{j-k} h_{l}(z_1, \dots, z_m) h_{j-k-l}(z_{m+1}, \dots, z_{d}) 
\ooo
Now we re-index $j = k+l$ for $l \geq0$. Note that since the index $k$ starts at $2$, the condition 
\begin{equation} \label{s+j} 
s+j \geq 2
\end{equation}
is always satisfied. Applying Lemma \ref{hh to he} completes the proof for $i \geq1$. 

For $i=0$, we use the same reasoning, but the index $j$ now starts at $1$. At $j=1$ the condition \eqref{s+j} is not satisfied if $s=0$; the values $j=1$ and $s=0$ give exactly one term in the double sum that evaluates to $-1$, so we add and subtract $-1$ to obtain the term 
\[
\mathbf{1}(s=0).
\] 
This completes the proof. 
\end{proof}

\begin{lemma} \label{powers of a}
The determinant 
 \[
 \det(J_{\vect{f}_m})(\vect{\alpha})
 \]
 has factors of 
 \[
 a_m^{-m}, \quad \text{ and } \quad a_{m-1}^0.
 \] 
\end{lemma}
\begin{proof}
From the definitions we check that 
\begin{equation} \label{entry}
J_{\vect{f}_m}(\vect{\alpha})_{i,j}
\end{equation}
has a factor of 
\[
a_{m-1}^{i-j}.
\]
Using the definition of determinant has a sum over the symmetric group $S_m$, we see by induction that each of the $m!$ terms in this sum has a factor of $a_{m-1}^0$.  

Similarly, we see that $\eqref{entry}$ has a factor of 
\[
\begin{cases}
a_m^{j-i-1} \text{ if } i  \text{ and } j \geq 2 \\ 
a_m^{1-i} \text{ if } i\geq 2, j=1 \\  
a_m^{j-3} \text{ if } i=1, j\geq 2 \\  
a_m^{-1} \text{ if } i=j=1 \\ 
\end{cases}.
\]
We see by induction that each of the $m!$ terms in the sum for the determinant has a factor of $a_m^{-m}$. This completes the proof. 
\end{proof}

\begin{definition}For integers $n,n_1,n_2 \geq 0$ and indeterminates $x_i,y_i$, define
\[
\mu_n(x_1, \ldots, x_{n_1}; y_1, \ldots, y_{n_2}) = \sum_{l=0}^n (-1)^l h_l(x_1, \ldots, x_{n_1}) e_{n-l}(y_1, \ldots, y_{n_2})
\] 
\end{definition}

Lemma \ref{powers of a} allows us to factor out powers of $a_{m-1}$ and $a_{m}$ from the determinant, yielding 
\[
\det(J_{\vect{f}_m})(\vect{\alpha})  a_m^m  (\prod_{n=1}^d z_n)^m= \det (M)
\]
where $M$ is the following matrix. 

\begin{definition}
For $1 \leq j \leq m$ and $2 \leq i \leq m$, 
\ppp
M_{1,j} &= \sum_{k=0}^{\max(m-j-1,0)} a_k \mu_{d-(2m-j)-k}(z_1, \ldots, z_m; z_{m+1}, \dots z_d)\\
M_{i,j} &= (\prod_{n=1}^m z_n)\sum_{c=1}^{i-1} h_{c-1}(z_{m+1}, \ldots, z_d) \sum_{k=0}^{\max(m-j-1,0)} a_k \mu_{d-(2m-j)-k+c}(z_1, \ldots, z_m; z_{m+1}, \dots z_d)\\ 
&- \mathbf{1}(j>i) e_{i-j-1}(\frac{1}{z_{m+1}}, \ldots, \frac{1}{z_d}) \prod_{n=m+1}^d z_n
\ooo 
\end{definition}

Next we apply row and column operations to $M$ to express $\det(M)$ as $\det(U+V)$ for the following matrices $U$ and $V$. 

\begin{definition}
\[
U_{i,j} = 
\begin{cases} 
\mu_{d-2m+j-i+1}(z_1, \ldots, z_m; z_{m+1}, \dots z_d)  &\text{ if } i=1 \\  
\mu_{d-2m+j-i+1}(z_1, \ldots, z_m; z_{m+1}, \dots z_d)(\prod_{n=1}^m z_n) &\text{ if } i >1\\ 
\end{cases}
\]
 
\[
V_{i,j}=
\begin{cases} 
- \mu_{i-j-1}(\frac{1}{z_1}, \ldots, \frac{1}{z_m}; \frac{1}{z_{m+1}}, \ldots, \frac{1}{z_d} ) \prod_{n=m+1}^d z_n &\text{ if } j < i \\  
0 & \text{ if } j \geq i
\end{cases}
\]

\end{definition}

\begin{lemma} \label{det M}
 
\[
\det(M) = \det(U+V)
\] 
\end{lemma}

\begin{proof}
 Let $\mathrm{row}(i)$ and $\mathrm{col}(j)$ denote the $i$th row and $j$th column of a matrix, respectively. First perform the successive column operations to $M$: 
starting with $j_1 = j_2-1$ and proceeding to $j_1=1$, perform
\[
 \mathrm{col}(j_1) \mapsto  \mathrm{col}(j_1)-a_{j_2-j_1} \mathrm{col}(j_2).
\] 
 Do this for each $j_2$ starting from $j_2=m-1$ and proceeding to $j_2 = 2$. 
 
 Such an operation removes all the terms in column $j_1$ of the form 
 \[
 a_{j_2-j_1} \mu_{\cdot} (z_1, \dots, z_m; z_{m+1}, \ldots, z_d).
 \]
Next perform the successive row operations to the resulting matrix: 
starting with $i_1 = i_2+1$ and proceeding to $i_1=m$, perform
\[
\mathrm{row}(i_1) \mapsto  \mathrm{row}(i_1)- h_{i_1-i_2}(z_{m+1}, \ldots, z_d) \mathrm{row}(i_2). 
\]
Do this for each $i_2$ starting from $i_2=2$ and proceeding to $2_2 = m-1$.
Such an operation removes all the terms in row $j_1$ of the form 
 \[
h_{i_1-i_2}(z_{m+1}, \ldots, z_d)  \mu_{\cdot} (z_1, \dots, z_m; z_{m+1}, \ldots, z_d).
 \]

Call the resulting matrix $M'$. We thus have that 
\[
M' = U + V'
\]
where we obtain the following formula for $V'$ by induction:
\[
V'_{i,j} = -(\prod_{n=m+1}^d z_n)\sum_{l=0}^{i-j-1} e_{l}(\frac{1}{z_{m+1}}, \ldots, \frac{1}{z_d}) \sum_{k=0}^l \left(\sum_{c \in C(k)} (-1)^{k+|c|}h_c(\frac{1}{z_{m+1}}, \ldots, \frac{1}{z_d})\right) \left(\sum_{c \in C(l-k)}  (-1)^{l-k+|c|}a_c\right). 
\] 
By identities \eqref{ec} and \eqref{hc} this is equal to 
\[
-(\prod_{n=m+1}^d z_n)\sum_{l=0}^{i-j-1} e_{l}(\frac{1}{z_{m+1}}, \ldots, \frac{1}{z_d}) \sum_{k=0}^l (-1)^{l-k}e_k(\frac{1}{z_{m+1}}, \ldots, \frac{1}{z_d})h_{l-k}(\frac{1}{z_1}, \ldots, \frac{1}{z_d}). 
\]
Now to the above sum we apply the identity 
\[
h_{k}(\frac{1}{z_1}, \ldots, \frac{1}{z_d})  = \sum_{n=0}^k h_{n}(\frac{1}{z_1}, \ldots, \frac{1}{z_m})h_{k-n}(\frac{1}{z_{m+1}}, \ldots, \frac{1}{z_d})
\]
and re-arrange the sum to obtain 
\[
\sum_{n_2=0}^{i-j-1} (-1)^{n_2}e_{i-j-1-n_2}(\frac{1}{z_{m+1}}, \ldots, \frac{1}{z_d}) \sum_{k=0}^{n_2}h_{n_2-k}(\frac{1}{z_1}, \ldots, \frac{1}{z_m})\sum_{n_1=0}^{k}(-1)^{n_1} e_{n_1}(\frac{1}{z_{m+1}}, \ldots, \frac{1}{z_d})h_{k-n_1}(\frac{1}{z_{m+1}}, \ldots, \frac{1}{z_d}).
\]
By identity \eqref{e h} the non-zero terms occur only when $k=0$. This yields the formula for $V$ and completes the proof. 
\end{proof} 

\begin{lemma} \label{mu z1}
Suppose $x_1 = y_{1}$. Then 
\[
\mu_n(z_1, \ldots, x_{n_1}; y_1, \ldots, y_{n_2}) = \mu_n(z_2, \ldots, x_{n_1}; y_2, \ldots, y_{n_2}). 
\] 
\end{lemma}

\begin{lemma} \label{lin combo}
Suppose $z_1 = z_{m+1}$. Then 
for any $1 \leq j \leq m$, 
\[
z_1M_{1,j}+ \sum_{i=2}^{m} e_{m-i}(\frac{1}{z_2}, \ldots, \frac{1}{z_m})M_{i,j} = 0. 
\]
\end{lemma} 

\begin{proof}
We claim 
\[
e_{m-i}(\frac{1}{z_2}, \ldots, \frac{1}{z_m}) U_{i,j} =  z_1 e_{i-1}(z_2, \ldots,z_m) \sum_{l=0}^{d-2m+j-i+1} (-1)^l h_l(z_2, \ldots, z_m) e_{d-2m+j-i+1 - l}(z_{m+2}, \ldots, z_d).
\] 
For $i\geq 2$, we apply Lemma \ref{mu z1} when $z_1 = z_{m+1}$ to obtain  
\ppp
e_{m-i}(\frac{1}{z_2}, \ldots, \frac{1}{z_m}) U_{i,j} &= e_{m-i}(\frac{1}{z_2}, \ldots, \frac{1}{z_m}) (\prod_{n=1}^m z_n) z\sum_{l=0}^{d-2m+j-i+1} (-1)^l h_l(z_2, \ldots, z_m) e_{d-2m+j-i+1 - l}(z_{m+2}, \ldots, z_d). \\ 
\ooo
Then use
\[
e_{m-i}(\frac{1}{z_2}, \ldots, \frac{1}{z_m}) (\prod_{n=1}^m z_n) = z_1e_{i-1}(z_2, \ldots, z_m)
\]
to prove the claim for $i \geq 2$. 
For $i=1$, we have also by Lemma \ref{mu z1}
\[
z_1 U_{1,j} = z_1 \sum_{l=0}^{d-2m+j} (-1)^l h_l(z_2, \ldots, z_m) e_{d-2m+j - l}(z_{m+2}, \ldots, z_d).
\]
This proves the claim for $i =1$. 
Using the claim we express 
\ppp
&z_1 U_{1,j} +\sum_{i=2}^m e_{m-i}(\frac{1}{z_2}, \ldots, \frac{1}{z_m}) U_{i,j} \\ 
&= z_1 \sum_{i=1}^m e_{i-1}(z_2, \ldots,z_m) \sum_{l=0}^{d-2m+j-i+1} (-1)^l h_l(z_2, \ldots, z_m) e_{d-2m+j-i+1 - l}(z_{m+2}, \ldots, z_d).
\ooo
We re-arrange this to obtain 
\[
z_1 \sum_{l=0}^{d-2m+j}(-1)^l e_{d-2m+j-l}(z_{m+2, \ldots, z_d})\sum_{k=0}^{l} (-1)^k h_{k}(z_2, \ldots, z_m)e_{l-k}(z_2, \ldots, z_m)  
\]
where we use for $k \geq m$
\[
e_{k}(z_2, \ldots, z_m) = 0.
\]
Applying identity \eqref{e h} shows that the term for $l=0$ is the only non-zero contribution
\begin{equation} \label{U result}
z_1 e_{d-2m+j}(z_{m+2}, \ldots, z_d).
\end{equation}

Next consider 
\begin{equation} \label{V sum}
z_1 V_{1,j} +\sum_{i=2}^m e_{m-i}(\frac{1}{z_2}, \ldots, \frac{1}{z_m}) V_{i,j}. 
\end{equation}
Now 
\[
V_{1,j}=0
\] 
and for $ i \geq 2$ we apply Lemma \ref{mu z1} to obtain 
\ppp
V_{i,j} &= -z_1(\prod_{n=m+2}^d z_n) \sum_{l=0}^{i-j-1} (-1)^l h_l(\frac{1}{z_2}, \dots, \frac{1}{z_m})e_{i-j-1-l}(\frac{1}{z_{m+2}}, \ldots, \frac{1}{z_d}) \\
&= -z_1 \sum_{l=0}^{i-j-1} (-1)^l h_l(\frac{1}{z_2}, \dots, \frac{1}{z_m})e_{d-(m+1)-(i-j-1-l)}(z_{m+2}, \ldots, z_d) 
\ooo 
Thus expression \eqref{V sum} is equal to 
\[
\sum_{i=2}^m - e_{i-1}(\frac{1}{z_2}, \dots, \frac{1}{z_m})z_1 \sum_{l=0}^{i-j-1} (-1)^l h_l(\frac{1}{z_2}, \dots, \frac{1}{z_m})e_{d-(m+1)-(i-j-1-l)}(z_{m+2}, \ldots, z_d)
\]
which we re-arrange to obtain 
\[
-z_1 \sum_{l=0}^{m-j-1}(-1)^l e_{d-(m+1)-(m-j-1-l)}(z_{m+2}, \ldots, z_d)\sum_{k=0}^l (-1)^k h_k(\frac{1}{z_2}, \dots, \frac{1}{z_m}) e_{l-k}(\frac{1}{z_2}, \dots, \frac{1}{z_m}). 
\]
Applying identity \eqref{e h} shows that the term for $l=0$ is the only non-zero contribution 
\[
-z_1 e_{d-2m+j}(z_{m+2}, \ldots, z_d)
\]
which cancels out \eqref{U result}. Note that when $j=m$ both terms are 0. This completes the proof. 
\end{proof}

Now we can prove the factorization of $\displaystyle \det( J_{\vect{f}} (\vect{\alpha}))$. 
\begin{theorem} 
\begin{equation} \label{factorization}
\det( J_{\vect{f}} (\vect{\alpha})) = \frac{\prod_{j=1}^m \prod_{i=m+1}^d (z_i - z_j)}{ (a_m\prod_{i=1}^d z_i)^m } 
\end{equation}
\end{theorem}
\begin{proof}
Lemma \ref{det M} implies that left side of equation \eqref{factorization} is a symmetric polynomial $P$ in $z_{m+1}, \ldots, z_d$ whose coefficients are symmetric Laurent polynomials in $z_1, \ldots, z_m$. 

We claim the total degree of $P$ in $z_{m+1}, \ldots, z_d$ is $(m,\ldots, m)$. Each $z_k$ for $m+1 \leq k \leq d$ appears in $U_{i,j}$ with exponent 1 or 0, and the only entry of $U$ that  contains the term 
\begin{equation} \label{term}
\prod_{n=m+1}^d z_n
\end{equation}
is $U_{1,m}$. The coefficient of this term in $U_{1,m}$ is 1. Each entry $V_{i,j}$ for $i>j$ contains the term \eqref{term}. In the expansion of $\det(U+V)$ as a sum over the symmetric group $S_m$, for a $\sigma \in S_m$ take a product of entries 
\[
\prod_{i=1}^m (U+V)_{\sigma(j),j}.
\]
We consider which $\sigma$ yield 
\begin{equation} \label{max term}
\prod_{n=m+1}^d z_n^m.
\end{equation}
Suppose $1 \leq j \leq m-1$. If $\sigma(j) \leq j$ for then $(U+V)_{\sigma(j),j}$ does not contain the term \eqref{term}. Thus suppose $\sigma(j)> j$. Let $j_0$ be the greatest $j$ such that $\sigma(j)>j+1$. Then $\sigma(\sigma(j_0)-1)\leq j$ and thus 
\[
(U+V)_{\sigma(\sigma(j_0)-1), \sigma(j_0)-1}
\]  
does not contain term \eqref{term}. Therefore the only $\sigma$ that yields the highest degree term \eqref{max term} is given by $\sigma(j) =j+1$ and $\sigma(m)=1$. The coefficient of the term \eqref{term} in $P$ is thus 
\ppp
(-1)^{m-1} \mathrm{sgn}  
&= (-1)^{m-1} (-1)^{m-1} \\ 
&= 1.
\ooo

Lemma \ref{lin combo} gives non-trivial linear combination of the rows of $U+V$ when $z_1 = z_{m+1}$. Therefore $P$ is equal to 0 when $z_{m+1}=z_1$ and so has a factor of 
\[
z_{m+1}-z_1.
\]
By symmetry $P$ has factors of 
\[
(z_i - z_j)
\]
for each pair $i,j$ with $1 \leq j \leq m$ and $ m+1 \leq i \leq d$. 
Thus 
\[
P = L(z_1, \ldots, z_m )\prod_{j=1}^m \prod_{i=m+1}^d (z_i - z_j)
\]
where $\displaystyle L(z_1, \ldots, z_m)$ is some Laurent polynomial. From the above claim this $L$ is equal to 1. This completes the proof. 
\end{proof}

\section{$V(0,2) = \mathrm{attractors}(0,2)$} \label{set eq}

\begin{definition} \label{m n X}
For integers $m,n \geq 0$ and a sequence $\vec{y} = (y(i))_{i=0}^{m-n}$ of indeterminates $y(i)$
define the $m \times n$ matrix $B(m,n,\vec{y})$ by 
\[
B(m,n,\vec{y})_{i,j} = \begin{cases} 
y(i-j) \text{ if } 0 \leq i-j \leq m-n \\ 
0 \text{ otherwise } 
\end{cases}.
\]

Given sequences of integers $\vec{m} =(m(i))_{i=1}^k$ and $\vec{n} =(n(i))_{i=1}^l$, denote the lengths $k$ and $l$ by $|\vec{m}|$ and $\vec{n}|$ respectively. 
Given also a matrix $X$ of sequences $\vec{x}_{i,j}$ for $1 \leq i \leq k$ and $ 1 \leq j \leq l$ with 
\[
X_{i,j} = \vec{x}_{i,j} = (x_{i,j}(h))_{h=0}^{m(i) - n(j)},
\]
define the block matrix $M( \bar{X})$ to be the block matrix whose $(i,j)$-th block is the matrix $B(m(i), n(j), \vec{x}_{i,j})$. 

Let $\bar{X}$ denote the triple $(\vec{m},\vec{n}, X)$. 

We will use the special sequences 
\begin{align*}
\vec{m}_d &= (1+{d \choose 2}+2-2i)_{i=1}^{d-1} \\ 
\vec{m}_d' & =(1+{d \choose 2}-2i)_{i=1}^{d-1} \\ 
\vec{n}_d &= (1+{d-1 \choose 2} +1-j)_{j=1}^{d-1}
\end{align*}

\end{definition} 

\begin{definition}
Denote the composition $c$ of an integer $h$ into $r$ non-negative integer parts by  
\[
c = (c(1), c(2), \ldots, c(r)).
\] 
Let $|c|$ denote $r$. Let $C(h,r)$ denote the set of such compositions. 
\end{definition}

\begin{definition} \label{d null vec v}

Given $\vec{m}, \vec{n}$ and $X$ as in Definition \ref{m n X}, suppose $|\vec{m}| +1= |\vec{n}|=d-1$. 
Fix an integer $j_0$ for $1 \leq j_0 \leq d-1$, let $S(j_0)$ denote the set of bijections from $[d-1] \backslash \{ j_0 \}$ to $[d-2]$. Note that we may identify a permutation $\tau_\sigma$ of $[d-2]$ with an element $\sigma \in S(j_0)$ via 
\[
\sigma(j) = \begin{cases} \tau(i) &\text{ if } j<j_0\\ 
\tau(j-1) & \text{ if } j>j_0.
\end{cases}
\]
 For $\sigma \in S(j_0)$, define $\#\mathrm{inv}(\sigma)$ to be the number of involutions in the sequence 

\[
(\sigma(1), \ldots, \sigma(j_0-1), \sigma(j_0+1), \ldots, \sigma(d-1)),
\]

so the signature 
\[
\mathrm{sgn}(\tau_\sigma) = (-1)^{\# \mathrm{inv}(\sigma)}.
\]

 For $\sigma \in S(j_0)$ and $ c \in C(l,d-2)$ for any $l$, define $\pi(c,\sigma, \bar{X})$ to be 
\[
\pi(c,\sigma, \bar{X}) = (-1)^{\#\mathrm{inv}(\sigma)}\prod_{i=1}^{d-2} x_{i,\sigma^{-1}(i)}(m(i) - n(\sigma^{-1}(i))-c(i))
\]
where $x_{i,j}(h)$ denotes 0 if $h >m(i)-n(j)$ or if $h<0$. 
Let $\vec{v}_j(\bar{X})$ denote the vector 
\[
\vec{v}_j(\bar{X}) = (v_j(l,\bar{X}))_{l=0}^{n(j)-1}.
\]
where 
\begin{align} \label{vik}
v_{j_0}(l,\bar{X}) & = (-1)^{j_0-1}\sum_{c \in C(j_0, n(j) - 1-l,d-1)} \sum_{\sigma \in S(j_0)} \pi(c, \sigma,\bar{X}) \\ 
& = (-1)^{j_0-1}\sum_{c \in C( n(j_0) - 1-l,d-2)} \det (L(j_0,c,X))
\end{align}
where the matrix $L(j_0,c,\bar{X})$ is the submatrix obtained from $M(\bar{X})$ by removing column block $j_0$, and then from row block $i$, removing all rows except for the $(m(i) - c(l))$-th row in that row block.  
Let $\vec{v}(\bar{X})$ denote the block vector (concatenation) of $\vec{v}_1(\bar{X}), \ldots, \vec{v}_{d-1}(\bar{X})$. 
\end{definition}

\begin{lemma} \label{l sum u}
Let $W(d)$ be the $(d-2)\times (d-1)$ matrix 
\[
W(d)_{i,j} =d-2-2i+j. 
\]
Let $\sigma$ be a surjection 
\[
\sigma \colon [d-1] \rightarrow [d-2].
\]
There are exactly two integers $u, v \in [d-1]$ such that 
\[
\sigma(u) = \sigma(v).
\]
Then 
\begin{equation} \label{sum u}
\sum_{1 \leq i \leq d-1, i \neq u } W(d) _{ \sigma(i),i}  =  {d -1 \choose 2} +1- u 
\end{equation}
\end{lemma}
\begin{proof} 
For $ 1 \leq u \leq d-1$, suppose $\sigma$ is of the form 
\[
\sigma(i) = \begin{cases} 
i \text{ if } i < u \\ 
i-1 \text{ if } i \geq u
\end{cases}.
\]
Then the left side of \eqref{sum u} becomes 
\begin{align*}
\sum_{i=1}^{u-1} (d-2-i) + \sum_{i=u+1}^{d-1} d-i &= -(u-1)+\sum_{i=2}^{d-1} (d-i) \\ 
&= -(u-1)+ {d-1 \choose 2}.
\end{align*}
This proves \eqref{sum u} for such $\sigma$. 
Now any $\sigma$ with 
\[
|\sigma^{-1}(\sigma(u))|=2
\]
can be obtained from the $\sigma$ described above by a sequence of transpositions of the values at the non-$u$ elements. Such a transposition, by construction of $W(d)$, does not change the sum \eqref{sum u}. This completes the proof. 
\end{proof}

\begin{lemma} \label{l null vec v}
Let $\bar{X}$ denote the triple $(\vec{m}_d', \vec{n}_d, X)$ for an arbitrary matrix of indeterminates.
Then the vector $\vec{v}(\bar{X})$ from Definition \ref{d null vec v} is in the null space of $M(\bar{X})$. 
\end{lemma}
\begin{proof}
We must prove that 
\[
\sum_{j=1}^{d-1}\sum_{k=0}^{n(j)-1}x_{i,j}(s-k)v_j(k,\bar{X})=0
\]
 for each $i$ and $s$. 
 
Fix $i$ and $s$. For each $k$ and $j$, expand $v_j(k)$ according to \eqref{vik} and choose a $c \in C(n(j) - 1-k,d-2)$ and $\sigma \in S(j)$. We thus have a set of terms indexed by the 4-tuple 
\[
(k,j,c,\sigma)
\]
where we assume
\begin{align}
0 &\leq s-k \leq m(i) - n(j)  \label{nc 1} \\ 
0 &\leq k \leq n(j)-1 \label{nc 2} \\ 
0 &\leq c(l) \leq m(l) - n(\sigma^{-1}(l)) \text{ for } 1 \leq l \leq d-2.  \label{nc 3}
\end{align}
for otherwise the corresponding term is 0. We construct a sign-changing involution by constructing indices $k', j', c'$ and $\sigma'$.

 Let 
 \[
 j' = \sigma^{-1}(i)
 \]
 and let 
 \[
 k'=s -  (m(i) - n(\sigma^{-1}(i)) - c(i)).
 \]
  Let $c'$ be the composition obtained from $c$ by setting 
\[
c'(i) = m(i) - n(j) -(s-k).
\] 
Let $\sigma' \in S(j')$ be determined as follows. Take the sequence 
\begin{equation} \label{sigma seq}
(\sigma(1), \ldots, \sigma(j-1), i, \sigma(j+1), \ldots, \sigma(r+1))
\end{equation}
and remove the entry of the form $\sigma(j')$, i.e. the other element that is equal to $i$. The resulting sequence is a permutation of $[d-2]$ which we identify with an element of $S(j')$. Note 
\begin{align*} 
\sigma'^{-1}(i) &= j \\ 
 \sigma^{-1}(i) &= j'\\ 
 c' &\in C(n(j') -1 - k', d-2).
 \end{align*}  

Assuming $(k,j,c,\sigma)$ satisfies conditions \eqref{nc 1}, \eqref{nc 2}, and \eqref{nc 3}, we prove that $(k',j',c',\sigma')$ does as well. For \eqref{nc 1},  the inequality 
\[
0 \leq s-k' \leq m(i) - n(j')
\]
follows from 
\[
0 \leq c(i)  \leq m(i) - n(\sigma^{-1}(i)) = m(i) - n(j').
\]

For \eqref{nc 3}, the only condition that is changed is when $l =i$, and 
\[
0 \leq c'(i) \leq m(i) - n(\sigma'^{-1}(i)))
\]
follows from 
\[
0 \leq s-k \leq m(i) - n(j) =  m(i) - n(\sigma'^{-1}(i)). 
\]

For \eqref{nc 2}, 
\[
0 \leq \sum_{l=1}^{d-2} c'(l) = n(j') - 1- k'
\]
which implies $k \leq n(j') -1$. To prove 
\begin{equation} \label{k' bound}
0 \leq k' ,
\end{equation}
consider the matrix $W(d)$ as in Lemma \ref{sum u}. The entry $W(d)_{i,j}$ is equal to $m(i) - n(j)$; if this number is non-negative, it is the maximum value of a composition part chosen for the $(i,j)$-th block that results in a non-zero term (i.e. condition \eqref{nc 3}). A negative entry $W(d)_{i,j}$ corresponds to an all-zero block matrix at position $(i,j)$. Condition \eqref{nc 3} implies $x_{l, \sigma'^{-1}(l)}(c'(l)) \neq 0$, so we must have
\[
c'(l) \leq W(d)_{l, \sigma'^{-1}(l)}.
\]
Summing $1 \leq l \leq d-2$, we have 
\begin{align}
n(j') - 1 -k' &= \sum_{l=1}^{d-2} c'(l)  \nonumber \\ 
   &\leq \sum_{l=1}^{d-2} W(d)_{l, \sigma'^{-1}(l)} \nonumber \\ 
   & = {d-1 \choose 2}+1 - j' \label{use sum u} \\ 
   & = n(j') - 1 \nonumber 
\end{align}
where we have applied Lemma \ref{sum u} at line \eqref{use sum u}. This proves \eqref{k' bound}. 

Now we prove that the involution is sign-reversing, i.e.. 
\begin{equation} \label{sign change}
(-1)^{j-1+ \#\mathrm{inv}(\sigma)} = -(-1)^{j'-1+ \#\mathrm{inv}(\sigma')} .
\end{equation}
The only elements in the sequence \eqref{sigma seq} that change their inversion status with respect to $i$ under the change from $\sigma$ to $\sigma'$ are those between the $j$-th and $j'$-th positions. There are $|j-j'|-1$ such positions. Suppose $k_1$ of these elements are inversions in $\sigma$ and $k_2$ are inversions in $\sigma'$. Therefore 
\begin{align*}
 \#\mathrm{inv}(\sigma) -  \#\mathrm{inv}(\sigma') &= k_1 - k_2\\ 
    &  \equiv k_1+k_2 \mod 2, 
\end{align*}
and 
\[
k_1+k_2 = |j-j'|-1.
\]
This proves \eqref{sign change} and completes the proof. 
\end{proof}

\begin{definition} 
With $\bar{X}$ as in Lemma \ref{l null vec v}, we call $v(\bar{X})$ the rank-independent null vector of $M(\bar{X})$. 
\end{definition}

\begin{definition} \label{d Xf}


For integers $0 \leq j \leq d$,
\[
\vec{x}_{j} = (x_{j}(l))_{l=0}^{d-j}
\]
where 
\[
x_{j}(l) = -t^{\lfloor \frac{j-1}{2} \rfloor}{l+\lfloor \frac{j-1}{2} \rfloor \choose \lfloor \frac{j-1}{2} \rfloor }\frac{a_{l+j}}{a_2},
\]
\[
t = \frac{a_0}{a_1},
\]
and 
\[
{ -1 \choose -1}
\]
denotes 1. 

Let $\emph{system}(f)$ denote the $M(\vec{m}_d, \vec{n}_d, X)$ where 
\[
X_{i,j} = \vec{x}_{2i-j}.
\]
Let $\mathrm{system}'(f)$ denote the matrix obtained from $\emph{system}(f)$ by removing the first row block. Then  $\mathrm{system}'(f)$ is satisfies the premises of Lemma \ref{l null vec v}; let $v(f)$ denote the corresponding rank-independent null vector of $\mathrm{system}'(f)$.
\end{definition} 

\begin{definition}
For an integer $1 \leq i \leq d-1$, $1 \leq r\leq d$,let
\begin{align*}
\mathrm{source}(i) &= (\lceil \frac{i+1}{2} \rceil,  -\lfloor \frac{i-1}{2} \rfloor ) \\ 
\mathrm{sink}(i, r) &=(i, d-i-r)\\ 
\end{align*}
denote points in $\mathbb{Z}^2$. 

Let $\emph{coord}_1$ denote the mapping 
\[
\emph{coord}_1 \colon \mathrm{sink}(i, r) \mapsto i.
\]
 \end{definition} 
 
 \begin{definition} \label{d graphs}
Fix integers $1 \leq i_0,j_0 \leq d-1$. Let $k \geq 0$ be an integer. 

For $c \in C(k,d-1)$, let $\vec{B}_{c,d}$ denote the sequence
\begin{align*} 
\vec{B}_{c, d} = (\mathrm{sink}(i,c(i)))_{i=1}^{d-1}, 
\end{align*}
and for $c \in C(k,d-2)$ let $\vec{B}_{i_0,c,d}$ denote the sequence 
\[
\vec{B}_{i_0, c, d} = (\mathrm{sink}(i+\mathbf{1}(i\geq i_0),c(i)))_{i=1}^{d-2}, 
\]

We will define sets of directed graphs. Let the vertex set for each graph be $\mathbb{Z}^2$. Let $E$ denote the set of directed edges of the form
\[
((i,j) \rightarrow  (i,j+1))
\]
or 
\[
((i,j) \rightarrow (i+1, j+1))
\]
for $i,j \in \mathbb{Z}$. 

For $c \in C(k,d-1)$, let $\mathrm{VD}(c)$ denote the set of directed graphs $G$ that consist of $d-1$ paths $G_1, \ldots, G_{d-1}$, such that path $G_l$ begins at $\mathrm{source}(l)$, and either

a) $G_l$ ends at a vertex $\mathrm{sink}(G_l) \in B_{ c,d}$ with each edge in $E$, or

b) $l=2i$ , $c(i) = d$, and $G_{l}$ consists of one directed edge to $\mathrm{sink}(G_{l}) = \mathrm{sink}(i,d)$; 
and the set of paths that satisfy a) are pairwise vertex disjoint.

Let $\mathrm{VD}(i_0, j_0, c)$ denote the set of directed graphs $G$ that consist of $d-2$ paths $G_1, \ldots, G_{d-2}$, such that path $G_l$ begins at  $\mathrm{source}(l + \mathbf{1}(l \geq j_0))$ and either

a) $G_l$ ends at a vertex $\mathrm{sink}(G_l) \in B_{c,d}$ with each edges in $E$, or

b) $l+\mathbf{1}(l \geq j_0) =2i$ for some integer $i$ , $c(i-\mathbf{1}(i \geq i_0)) = d$, and $G_{l}$ consists of one directed edge to $\mathrm{sink}(G_{l}) = \mathrm{sink}(i,d)$; 
and the set of paths that satisfy a) are pairwise vertex disjoint.

For $G \in \mathrm{DG}(i_0,j_0,c)$ or $G \in \mathrm{DG}(c)$, let $|G|$ denote the number of paths in $G$.

Let $\#\mathrm{inv}(G)$ denote the number of inversions in the sequence
\[
(\emph{coord}_1(\mathrm{sink}(G_l))_{l=1}^{|G|}.
\]
Let $\epsilon(i,j)$ denote the integer 
\[
\epsilon(i,j) = \lfloor \frac{j}{2} \rfloor - (i-1)+\sum_{l=1}^{d-1} (l-1 - \lfloor \frac{l}{2} \rfloor)
\]
Define 
\[
w(c) = \prod_{ l = 1}^{|c|} (\frac{-a_{d-c(l)}}{a_2})
\]

\end{definition}

\begin{lemma} \label{l vd paths}

For $1 \leq j_0 \leq d-1$,
\begin{equation} \label{vd paths}
v_{j_0}(n(j_0) - k, f) =(-1)^{j_0-1} (\frac{a_0}{a_1})^{\epsilon(1,j_0)} \sum_{c \in C(k,d-2)} w(c)\sum_{G \in \mathrm{VD}(1,j_0,c)} (-1)^{\# \mathrm{inv}(G)}
\end{equation}
\end{lemma} 
\begin{proof}

Let
\[
\bar{X} = (\vec{m}_d', \vec{n}_d, X)
\]
where 
\[
X_{i,j} = \vec{x}_{2+2i-j}.
\]

It follows from Definitions \ref{d null vec v} and \ref{d Xf} of $v_{j_0}(m(1) - n(j_0) - l, f)$, that, for a given $c$, we may factor out of $\det(L(j_0,c,\bar{X}))$ 
\[
\det( L(c,j_0,\bar{X})) = (\frac{a_0}{a_1})^{\epsilon(1,j)} w(c) \det( L'(c,j_0,\bar{X})) 
\] 
where 
\begin{align} 
L'(c,j_0,\bar{X})_{i,j} &= {m(i)-n(j) - c(i)+\lfloor \frac{2+2i-j-1}{2} \rfloor \choose \lfloor \frac{2+2i-j-1}{2} \rfloor }
\end{align}
and the above binomial coefficient is equal to 
\[
 \begin{cases} 
1 &\text{ if } j=2+2i \text{ and } c(i) = d  \label{1}\\ 
 \#\{\text{paths in $E$ from source($j$ + \textbf{1}($j \geq j_0$) to sink($i+1$)} \} &\text{ otherwise } 
\end{cases}.
\]
This first case above corresponds to the case $b)$ in Definition \ref{d graphs}. Applying the proof of the Lindstr\"om-Gessel-Viennot Lemma \cite{Gessel Viennot} and \cite{Lind}, we obtain that 
\[
\det (L'(c,j_0,\bar{X})) = \sum_{G \in \text{VD}(1,j_0, c)} (-1)^{\# \mathrm{inv}(G)};
\]
this sum includes graphs $G$ that contain intersecting paths $G_{l_1}$ and $G_{l_2}$, where $G_{l_1}$ is one edge from source($2i$) to sink($i,d)$, for some $i$, and $G_{l_2}$ begins at source($2i+1$) and passes through the point source($2i$). Such a graph is not canceled out by another in the sum, because, according to the Lindstr\"om-Gessel-Viennot involution, this graph maps to a $G'$ with a path 
\[
\mathrm{source}(2i+1) \rightarrow \mathrm{source}(2i) \rightarrow \mathrm{sink}(i,d),
\]
and such paths are not counted in $L'(c,j_0,\bar{X})$. This completes the proof. 
\end{proof}

\begin{definition}

From the definitions of the auxiliary functions we have the following formulas:
\begin{align*}
f_{0,2}(x_0,x_1) -x_0&= \sum_{i=-1}^{d-1}\sum_{j=0}^{\lfloor \frac{i}{2} \rfloor}  -\frac{a_{i+1}}{a_2}((\frac{-a_0}{a_2})(\frac{-a_1}{a_2})^{-1} x_1)^{j} x_0^{i-2j}{i-j \choose j}\\ 
f_{1,2}(x_0, x_1) -x_1&= \sum_{i=-2}^{d-2}\sum_{j=-1}^{\lfloor \frac{i}{2} \rfloor}  -\frac{a_{i+2}}{a_2}((\frac{-a_0}{a_2})(\frac{-a_1}{a_2})^{-1} x_1)^{j} x_0^{i-2j}x_1{i-j \choose j}.
\end{align*}

Define the function
\[
P(f, m)(z) = \prod_{S \subset [d], |S|=m} (z - \sum_{i \in S}z_i)
\]
\end{definition}

 \begin{definition} \label{d support}
 For a polynomial $p(x_0, \ldots, x_{m-1})$, define a monomial support of $p$ to  be a set of $m$-tuples of integers such that if an $m$-tuple $\vec{\beta}$
 \[
\harpoon{\beta} = (\beta_i)_{i=0}^{m-1} 
 \] 
 is not in a monomial support of $p$, then the coefficient of 
 \[
 \prod_{i=0}^r x_i^{\beta_i}
 \]
 in $p$ is 0.   
 
Let $d\geq 2$ be an integer. Let $p_0(x_0, x_1)$ and $p_1(x_0, x_1)$ be two polynomials such that a monomial support of $p_0$ is 
\[
\{ (i_0,i_1) \colon {0 \leq i_0 \leq d -1-2i_1, 0 \leq i_1} \}
\]
and a monomial support of $p_1$ is 
\[
\{ (i_0,i_1) \colon {0 \leq i_0 \leq d -2i_1, 0 \leq i_1} \}.
\]
We say that these polynomials have $(d,2)$-\emph{support}.
 \end{definition}

\begin{lemma} \label{l exists g}
Let $d\geq 2$ be an integer. Then there exist  polynomials $p_0,p_1$, and $g$ with coefficients in $R$ such that 
\begin{equation} \label{exists g}
p_0(x_0, x_1) (f_{0,2}(x_0, x_1)-x_0) + p_1(x_0, x_1) (f_{1,2}(x_0, x_1)- x_1)  = g(x_0)
\end{equation}
where $g(x_0)$ has monomial support $[0,{ d \choose 2}]$ and whose coefficient of $x_0^{d \choose 2}$ is 
\begin{equation} \label{lc}
(-\frac{a_d}{a_2})^{d-1}  (\frac{a_0}{a_1})^{\epsilon(1,1)}.
\end{equation}

\end{lemma} 
\begin{proof} 
For integers $0\leq j \leq \lfloor \frac{d-1}{2} \rfloor$ and $N \geq 0$, the set 
\[
\{ (i_0,i_1) \colon 0 \leq i_0 \leq N+d-1+2-2i_1 \text{ and } 0 \leq i_1\}
\]
is a monomial support of both 
\begin{equation} \label{C0}
x_1^{j=1}(\sum_{k=0}^{N+2-2j} C_{0,j,k} x_0^k) (f_{0,2}(x_0, x_1)-x_0)
\end{equation}
and 
\begin{equation} \label{C1}
x_1^{j-1}(\sum_{k=0}^{N-1+2-2j} C_{1,j,k} x_0^k) (f_{1,2}(x_0, x_1)-x_1)
\end{equation}
where $C_{0,j,k}, C_{1,j,k} \in R$. For $1 \leq i,j \leq d-1$, we construct a block matrix such that the $(h,k)$-th entry in the $(i,j)$-th block is the coefficient of 
\[
x_1^{i-1}x_0^{h-1}
\] 
in the expansion of 
\[
x_1^{\lfloor \frac{j-1}{2} \rfloor} x_0^k (f_{(j-1) \mod 2,2}(x_0, x_1) -x_{(j-1) \mod 2} ),
\]
where the $(i,j)$-th block has dimensions $m(i) \times n(j)$. Choosing $N = {d-1 \choose 2}$, this block matrix is equal to system$(f)$. Thus removing the first row block from $\text{system}(f)$ yields the matrix $\text{system}'(f)$, and the entries of the null vector $\vec{v}(f)$ specified by Lemma \ref{l vd paths} become the coefficients $C_{0,j,k}$ and $C_{1,j,k}$ of expressions \eqref{C0} and \eqref{C1}. This proves that the polynomial $g(x_0)$ has the stated monomial support, because the only non-zero monomials occur from the rows of the first block of system$(f)$, and these monomials have no power of $x_1$. 

Now we prove that $g(x_0)$ has degree ${d \choose 2}$. The leading coefficient of $g(x_0)$ is then 
\[
\text{system}(f)_{1+{d\choose 2} } \cdot \vec{v}(f)
\]
where $\text{system}(f)_{1+{d\choose 2} }$ is the $(1+{d\choose 2} ) $-th row of $\text{system}(f)$. This row has all zero entries, except for the entry $-\frac{a_d}{a_2}$ at column $1 + {d-1 \choose 2}$ in block (1,1). Therefore the leading coefficient of $x_0^{{d \choose 2}}$ in $g(x_0)$ is 
\[
-\frac{a_d}{a_2} v_{1}({d -1\choose 2},f)
\]
By assumption 
\[
\frac{a_d}{a_2}  \neq 0.
\]
To evaluate $ v_{1}({d -1\choose 2},f)$ we apply Lemma \ref{l vd paths}. There is only one $c \in C(1,0, d-2)$, the $c$ with $c(i) =0$ for all $i$. And for this $c$, there is only one directed graph $G \in \mathrm{VD}(c,d)$. This $G$ has 
\[
\#\mathrm{inv}(G) = 0.
\]   
Therefore the coefficient of $x_0^{{d \choose 2}}$ in $g(x_0)$ is \eqref{lc}. This completes the proof.  
\end{proof}

\begin{theorem} \label{t v02}
Let $f(z) $ 
\[
f(z) = \sum_{=0}^d a_i z^i \in \mathbb{C}[z]
\]
be a degree $d \geq 2$ polynomial with zeros $z_1, \ldots, z_d$ such that $a_1\neq 0, a_2 \neq 0$. Recall the set $V(0,2)$ 
\begin{equation} \label{V0}
V(0,2) = \{z_i + z_j \colon 1 \leq i < j \leq d \}.
\end{equation}
Suppose the point $(\alpha_0, \alpha_1)$ is an attractor point for \emph{NRS(2)} applied to $f(z)$ with arbitrary starting point. Then 
\[
\alpha_0  \in V(0,2)
\]
\end{theorem}
\begin{proof}
As NRS(2) is the two-dimensional Newtown-Raphson-Simpson method applied to the system of polynomials $f_{0,2}(x_0, x_1) - x_0$ and $f_{1,2}(x_0, x_1)- x_1$, the set of attractor points of NRS(2) is the set of common zeros of that system.
Thus we need to prove that if
\begin{equation}\label{common zero}
 f_{0,2}(\alpha_0, \alpha_1)  - \alpha_0 = f_{1,2}(\alpha_0, \alpha_1) - \alpha_1 = 0
\end{equation}
then $\alpha_0 \in V(0,2)$.

Suppose condition $C$: all the numbers $z_i+z_j$ are distinct and all the $z_i$ are non-zero. Let $V_0$ denote the set of all such $\alpha_0$ satisfying \eqref{common zero}. With $g(x_0)$ from Lemma \ref{l exists g}, then we must have
\[
g(\alpha_0) = 0.
\]
From Theorem \ref{t alpha}, there are at least ${ d \choose 2}$ choices for $\alpha_0$. Again from Lemma \ref{l exists g}, $g(x_0)$ has degree ${ d \choose 2}$. This implies that, assuming condition $C$, that $\alpha_0 \in V(0,2)$ and that 
\begin{equation} \label{g P}
g(x_0) = (-\frac{a_d}{a_2})^{d-1}  (\frac{a_0}{a_1})^{\epsilon(1,1)} P(f,2)
\end{equation}

Now equating the coefficients of the monomials in $x_0$ and $x_1$ on both sides of equation \eqref{exists g} and multiplying through by $a_1^{\epsilon(1,1)} a_2^{d-1}$, we obtain polynomial equations in the variables $z_i$ that are true whenever condition $C$ is true. But condition $C$ describes an open set in $\mathbb{C}^d$, so by analytic continuation these equations must be true on all of $\mathbb{C}^d$. Therefore \eqref{g P} holds for all choices of $z_i$, and therefore \eqref{V0} does as well. This completes the proof. 
\end{proof}

\section{Related results} \label{related}

\begin{lemma} \label{l all paths} For an integer $0 \leq k \leq {d \choose 2}$, 
\begin{equation} \label{all paths}
\emph{system}(f)_{1+{d \choose 2} - k} \cdot v (X\mid_f) = (\frac{a_0}{a_1})^{\epsilon(1,1)} \sum_{c \in C(k,d-1)} w(c) \sum_{G \in \mathrm{VD}(c)} (-1)^{\# \mathrm{inv}(G)}.
\end{equation}
\end{lemma} 
\begin{proof} 
The only non-zero blocks in the first row of blocks in $\text{system}(d)$ are the first and second.

Consider the first block. From Lemma \ref{l vd paths}, $v_1( {d-1 \choose 2}-l, \bar{X})$ consists of a sum over $c \in C(1,l,d)$ and a sum over directed graphs $G \in \text{VD}(1,1, c)$. Multiplying this $v_1( {d-1 \choose 2}-l, \bar{X})$ with the entry 
\[
\text{system}(f)_{1+{d \choose 2} - k, 1+{d-1 \choose 2}-l} = -\frac{a_{d-k+l}}{a_2}
\]
corresponds to extending $c$ to an element  $c' \in C(k,d)$ with $c'(1) = k-l$ and adjoining to $G$ the path from $\mathrm{source}(1)$ to $\mathrm{sink}(1,k-l)$, creating $G' \in \text{VD}(c')$. 

Consider the second block. From Lemma \ref{l vd paths}, $v_2( {d-1 \choose 2}-1-l) \mid_f$ consists of a sum over $c \in C(2,l,d)$ and a sum over directed graphs $G  \in \text{VD}(1,2, c)$.  Multiplying this $v_2( {d-1 \choose 2}-1-l, \bar{X})$ with the entry 
\[
\text{system}(f)_{1+{d \choose 2} - k, 1+2{d-1 \choose 2} -l} =  \mathbf{1}(k-l=d)(\frac{a_0}{a_1})^{-1}(- \frac{a_{0}}{a_2})
\]
corresponds to extending $c$ to an element in $c' \in C(k,d)$ with $c'(1) = d$ and adjoining to $G$ the path from source(2) to sink(1,$d$), creating $G' \in \text{VD}(c')$. By definition, terms in $v_2(X \mid_f)$ have one more factor of $-1$ than do terms in  $v_2(X \mid_f)$; this -1 corresponds to the one more inversion in $G'$ than in $G$ created by adjoing the path. And since
\[
\epsilon(1,2) -1 = \epsilon(1,1)
\]
the power of $\frac{a_0}{a_1}$ agrees with that in \eqref{all paths}. 

The set $ \bigcup_{c \in C(k,d-1)} \text{VD}(c)$ is equal to the set of the $G'$ created above. This completes the proof. 
\end{proof}

\begin{definition} 

Let $\mathrm{DSG}(d)$ denote the set of directed graph $\mathcal{G}$ on $d$ vertices 

\[
V=\{1, \ldots, d\}
\]
such that there are no parallel edges and no 1- or 2-cycles. we call such graphs \emph{directed simple graphs}. 

Denote a directed edge $(i,j)$ as an edge going from $i$ to $j$ and denote the edge set of $\mathcal{G}$ as $E(\mathcal{G})$. Define
\[
w(\mathcal{G}) = \prod_{(i,j) \in E(\mathcal{G})} z_i.
\]

\end{definition}

\begin{theorem} \label{t simple graph gen}
\begin{equation} \label{simple graph gen}
\sum_{\mathcal{G} \in \mathrm{DSG}(d), |E(\mathcal{G})| = l} w(\mathcal{G})= \sum_{c \in C(l, d-1)}e_c(\vec{z})\sum_{G \in \mathrm{VD}(c)} (-1)^{\#\mathrm{inv}(G)}
\end{equation}
\end{theorem} 
\begin{proof} 
In the right side of equation \eqref{all paths}, removing the powers of $\frac{a_0}{a_1}$ and the denominators of $a_2$ yields the right side \eqref{simple graph gen}. Lemmas \ref{l exists g}, \ref{l all paths} and Theorem \ref{t v02} show that this expression is the coefficient of $(-1)^lz^{{d \choose 2}-l}$ in $P(f,2)(z)$. And this coefficient is equal to the left side of \eqref{simple graph gen}, for the edge $(i,j)$ of $\mathcal{G}$ corresponds to choosing $z_i$ from the factor 
\[
z - (z_i  +z_j).
\]
This completes the proof.
\end{proof}

\subsection{A formula for the Newton series of an arbitrary product of binomial coefficients} \label{ss bin prod}
This section proves Theorem \ref{t su}, which arises from counting graphs in the sum on the left side of \eqref{simple graph gen}.
 \begin{lemma} \label{l bin sum} 
 For non-negative integers $l, a$ and $b$,
 \begin{equation} \label{bin sum}
\sum_{i=1}^l {a \choose i} {b \choose l-i} = {a+b \choose l}.
 \end{equation} 
 \end{lemma} 
 \begin{proof}
There are ${a+b\choose l}$  subsets of size $l$ of the set $[a+b]$. There are ${a \choose i} {b \choose l-i}$ subsets $U$ of size $l$ such that $i$ elements of $U$ are also in $[a]$. Summing over all $i$ yields equation \eqref{bin sum}. This proves the lemma.  
 \end{proof}
 
 \begin{lemma} \label{l rec}
  For integers $m, n \geq 0$, let $x_i$ and $y_j$ be integers for $1\leq i \leq m$ and $1\leq j \leq n$ such that 
  \[
  \sum_{i=1}^m x_i = \sum_{j=1}^n y_j = A
  \]
  for some integer $A$.
   Then 
  \begin{equation} \label{rec}
  \sum_{ \substack{y_j=\sum_{i=1}^m a_{i,j} \\ x_i =\sum_{j=1}^n a_{i,j} } } \frac{1}{\prod_{i=1}^m \prod_{j=1}^n a_{i,j}!} = \frac{A!}{\prod_{i=1}^m  x_i! \prod_{j=1}^n  y_j!}.
  \end{equation}
 \end{lemma}

\begin{proof} 
We use induction on $m$.
Suppose $m=1$. Then $x_1=\sum_{j=1}^n y_j$ and equation \eqref{rec} becomes 
\[
\frac{1}{\prod_{j=1}^n y_j!} = \frac{x_1!}{x_1! \prod_{j=1}^n y_j!}.
\] 
Likewise for $n=1$.

We claim \eqref{rec} is true for $m=2$ and all $n$. We use induction on $n$. First suppose $m=n=2$. This case follows directly from Lemma \ref{bin sum}. 

Next assume \eqref{rec} is true for $m=2$ and some $n \geq 2$. Then 
\begin{align*}
  \sum_{ \substack{y_j=\sum_{i=1}^2 a_{i,j} \\ x_i =\sum_{j=1}^{n+1} a_{i,j} } } \frac{1}{\prod_{i=1}^2 \prod_{j=1}^{n+1} (a_{i,j}!)} &=   
  \sum_{y_{n+1}=a_{1,n+1}+ a_{2,n+1} }\frac{1}{a_{1,n+1}! a_{2,n+1}!}\sum_{ \substack{y_j=\sum_{i=1}^2 a_{i,j} \\ x_i -a_{i,n+1}=\sum_{j=1}^{n} a_{i,j} } } \frac{1}{\prod_{i=1}^2 \prod_{j=1}^{n+1} (a_{i,j}!)}. \\ 
     &\text{ Applying the induction hypothesis yields } \\ 
  &= \sum_{y_{n+1}=a_{1,n+1}+ a_{2,n+1} }\frac{1}{a_{1,n+1}! a_{2,n+1}!} \frac{(x_1+x_2 -y_{n+1})!}{(x_1 -a_{1,n+1})! (x_2 -a_{2,n+1})!\prod _{j=1}^n y_j!}\\
  &=   \frac{(x_1+x_2 -y_{n+1})!}{x_1! x_2!\prod _{j=1}^n y_j!}  \sum_{a_{1,n+1}=0 }^{y_{n+1}} {x_1 \choose a_{1,n+1}} {x_2  \choose y_{n+1} -a_{1,n+1} } \\
   &\text{ Applying Lemma \ref{l bin sum} with $l = y_{n+1}, a=x_1, b=x_2$ yields } \\ 
  &= \frac{(x_1+x_2)!}{x_1 ! x_2 ! y_{n+1}!\prod _{j=1}^n y_j!}
\end{align*}
This complete the claim for $m=2$. Likewise for $n=2$.

Now for some $m_0 \geq 2$, suppose \eqref{rec} is true for $m = m_0$ and all $n$. Then 
\[
  \sum_{ \substack{y_j=\sum_{i=1}^{m_0+1} a_{i,j} \\ x_i =\sum_{j=1}^{n} a_{i,j} } } \frac{1}{\prod_{i=1}^{m_0+1} \prod_{j=1}^{n} (a_{i,j}!)} =   
  \sum_{x_{m_0+1}=\sum_{j=1}^n a_{m_0+1,j}}\frac{1}{\prod_{j=1}^n a_{m_0+1,j}!}\sum_{ \substack{y_j-a_{m_0+1,j}=\sum_{i=1}^{m_0} a_{i,j} \\ x_i =\sum_{j=1}^{n} a_{i,j} } } \frac{1}{\prod_{i=1}^{m_0} \prod_{j=1}^{n} (a_{i,j}!)} 
  \] 
   \text{ Applying the induction hypothesis yields }
\begin{align*}
 & = \sum_{x_{m_0+1}=\sum_{j=1}^n a_{m_0+1,j}} \frac{1}{\prod_{j=1}^n a_{m_0+1,j}! )}\frac{((\sum_{j=1}^n y_j)-x_{m_0+1})!}{\prod_{i=1}^m x_i! \prod_{j=1}^n (y_j - a_{m_0+1,j})! }\\ 
  &= \frac{((\sum_{j=1}^n y_j)-x_{m_0+1})!}{\prod_{i=1}^m x_i! } \sum_{x_{m_0+1}=\sum_{j=1}^n a_{m_0+1,j}} \frac{1}{ \prod_{j=1}^n  a_{m_0+1,j}!  (y_j - a_{m_0+1,j})!} \\ 
  \end{align*}
   \text{ Applying the result for $m$ or $n=2$ yields }\
   \begin{align*}
  &=\frac{((\sum_{j=1}^n y_j)-x_{m_0+1})!}{\prod_{i=1}^m x_i! } \frac{(\sum_{j=1}^n y_j)!}{x_{m_0+1}!(\sum_{j=1}^n y_j)-x_{m_0+1})! \prod_{j=1}^n y_j! } \\ 
  &= \frac{(\sum_{j=1}^n y_j)!}{\prod_{i=1}^{m_0+1} x_i!  \prod_{j=1}^n y_j! }.
  \end{align*} 
  This completes the proof. 
\end{proof} 

\begin{definition} 
For an integer $r$, let $\nu$ 
\[
 \nu= (\nu(1), \ldots, \nu(r))
\]
denote an $r$ tuple of integers. Let $N(r)$ denote the set of such $\nu$.  Let $s$
\[
s = (s(U))_{U \subseteq [r]}
\]
 denote a $(2^r -1)$-tuple of integers $s(U)$ indexed by the non-empty subsets $U$ of $[r]$. Given $\nu \in N(r)$ and an integer $k$, let $S(\nu,k)$ denote the set of all $s$ such that for each $i$, $1 \leq i \leq r$,
 \[
\nu(i) = \sum_{i \in U \subseteq [r]} s(U)
 \] 
and 
\[
k = \sideset{}{'}\sum_{U \subseteq [r]} s(U).
\] 

For integers $0 \leq j <r$, let $t = t(V)_{\emptyset \neq V \subseteq [r-j]}$be a $(2^{r-j}-1)$-tuple of integers indexed by the nonempty subsets of $[r-j]$. Suppose $0 \leq i \leq r-j$. Then define $\nu(t,i)$ 
\[
\nu(t,i) = \nu(i) - \sum_{i \in V} t(V)
\]  
\end{definition} 

\begin{theorem}\label{t su} For integers $d\geq0$ and $r \geq 1$ and with $\nu \in N(r)$,
\begin{equation} \label{su}
\sum_{k=0}^\infty (d)_k \sum_{s \in S(\nu,k)}  \frac{1}{\sideset{}{'}\prod_{U\subseteq [r]}s(U)!} = \prod_{i=1}^r {d \choose \nu(i)}.
\end{equation}
\end{theorem} 
\begin{proof} 
We use induction on $r$. For $r=1$, the only non-zero term from the left of \eqref{su} occurs when $s(\{1\}) = k = \nu(1)$, yielding 
\[
\frac{(d)_{\nu(1)}}{\nu(1)!} = {d \choose \nu(1)}.
\] 
Assume \eqref{su} is true for some $r \geq 1$. Let $C(a,b)$ denote the set of compositions of the integer $a$ into $b$ possibly nonzero parts. Then 
\begin{align} \label{su ind}
 \sum_{k=0}^\infty (d)_k \sum_{s \in S(\nu,k)}  \frac{1}{\prod_{U\in [r+1]}s(U)!} =  \sum_{t \in C(\nu(r+1), 2^r)}\frac{1}{\prod_{V \subseteq [r]} t(V)!}\sum_{k=0}^\infty (d)_k \sum_{s \in S(\nu(t),k-\nu(r+1))}  \frac{1}{\prod_{U\subseteq [r]}s(U)!} 
\end{align}
the outer sum on the right is a sum over compositions $t$, where $t$ is an $2^r$-tuple of integers 
\[
t = (t(V))_{V \subseteq [r]}
\]
indexed by all of the subsets $V \subseteq [r]$ such that 
\[
\sum_{V \subseteq [r]} t(V) = \nu(r+1),
\]
where the variable $t(V)$ on the right side corresponds to $s(\{r+1\} \cup V)$ on the left side. 
Expressing 
\[
(d)_k = (d)_{\nu(r+1)} (d-\nu(r+1))_{k-\nu(r+1)}, 
\]
the right side of \eqref{su ind} becomes 
\[
\sum_{t \in C(\nu(r+1); 2^r)}\frac{(d)_{\nu(r+1)}}{\prod_{V \subseteq [r]} t(V)!}\sum_{k=0}^\infty (d-\nu(r+1))_{k-\nu(r+1)} \sum_{s \in S(\nu(t),k-\nu(r+1))}  \frac{1}{\prod_{U\subseteq [r]}s(U)!} 
\]
which, after applying the induction hypothesis, yields 
\[
(d)_{\nu(r+1)}\sum_{t \in C(\nu(r+1); 2^r)}\frac{1}{\prod_{V \subseteq [r]} t(V)!} \prod_{i=1}^r {d - \nu(r+1) \choose \nu(t,i)}.
\]
In the above sum we express 
\[
t(\emptyset) = \nu(r+1)-\sideset{}{'}\sum_{V \subseteq [r]} t(V)
\]
to obtain
\begin{align} \label{su claim}
&\sum_{t \in T(2^r-1)} \frac{1}{(\nu(r+1)-\sideset{}{'}\sum_{V \subseteq [r]} t(V))! \sideset{}{'}\prod_{V \subseteq [r]} t(V)!} \prod_{i=1}^r {d - \nu(r+1) \choose \nu(t,i)}
\end{align}
where $T(2^r-1)$ denotes the set of $(2^r-1)$-tuples of integers with each tuple indexed by the set of nonempty subsets of $[r]$. We claim that \eqref{su claim} is equal to 
\[
 \frac{1}{\nu(r+1)!}\prod_{i=1}^r {d \choose \nu(i)}.
\]

To \eqref{su claim} we apply $r$ steps, where each step consists of a change of variables and an application of Lemma \ref{rec}. We denote the $j$-the set of variables by $t_j(V)$ where $\emptyset \neq V\subseteq [r-j]$ and let 
\[
t_j = (t_j(V))_{\emptyset \neq V \subseteq [r-j]}
\]
 We claim that after the $j$-th step we obtain that \eqref{su claim} is equal to expression $E(j)$ defined by 
\begin{equation} \label{su claim j}
E(j) =\prod_{i=r-j+1}^{r} {d \choose \nu(i)}\sum_{t_j \in T(2^{r-j}-1)} \frac{1}{(\nu(r+1)-\sideset{}{'}\sum_{V \subseteq [r-j]} t_j(V))! \sideset{}{'}\prod_{V \subseteq [r-j]} t_j(V)!} \prod_{i=1}^{r-j} {d - \nu(r+1) \choose \nu(t_j,i)}.
\end{equation}
For $j=0$, expression \eqref{su claim j} is equal to \eqref{su claim}, with $t_0$ denoting $t$. We therefore describe the $(j+1)$th step. Define the change of variables 
\[
t_{j}(V ) \mapsto t_{j+1}(V)  - t_{j}(V \cup \{ r-j \}) 
\]
for each $V \subseteq [r-j-1], V \neq \emptyset$. After this change of variables, all variables are of the form either $t_{j+1}(V)$ for $V \subseteq [r-j-1], V \neq \emptyset$ or $t_j(V \cup \{ r-j\}), V \subseteq [r-j-1] $. The $(j+1)$-th step then consists of summing over the variables $t_j(V)$ by applying Lemma \ref{rec}.
 
 Now for $1\leq i \leq r-j-1$ we have
 \[
\nu(t_j,i) = \sum_{}\nu(t_{j+1},i)
 \]
 because $i \in V\cup \{r-j\}$ if and only if $i \in V$.  Therefore the dependence on $t_j(V)$ in the summand of \eqref{su claim j} is 
 \begin{align}
 &\frac{1}{(\nu(r+1) - t(\{r-j\})-\sideset{}{'}\sum_{V \subseteq [r-j-1]}t_{j+1}(V))!}\times \\
 &\frac{1}{t_j(\{r-j\})!}\times \\
 &\sideset{}{'}\prod_{V\subseteq [r-j-1]}\frac{1}{(t_j(V \cup \{r-j\})!} \times\\
 &\sideset{}{'}\prod_{V\subseteq [r-j-1]}\frac{1}{ (t_{j+1}(V) - t_j(V \cup \{r-j\})!}\times\\
 &\frac{(d-\nu(r+1))!}{(\nu(r-j) - \sum_{V \subseteq [r-j-1]}t(V\cup \{r-j\}))!(d - \nu(r+1)- \nu(r-j)+\sum_{V \subseteq [r-j-1]} t(V\cup \{ r-j\}) )!} 
 \end{align}
 
 Now we sum over all the variables of the form $t(V\cup \{r-j\})$ for $V\subseteq [r-j-1]$ by applying Lemma \ref{rec} with $m=2$ and $n=2^{r-j-1}+1$;  and 
 \begin{align*}
 a_{1,1} &= t_j(\{r-j\})\\
 a_{2,1} &= \nu(r+1) - t(\{r-j\})-\sideset{}{'}\sum_{V \subseteq [r-j-1]}t_{j+1}(V))\\
 a_{1,i} &= t_j(V_i \cup \{r-j\})\\
 a_{2,i} &= (t_{j+1}(V_i) - t_j(V_i \cup \{r-j\})!\\
 a_{1,2^{r-j-1}+1}&= (\nu(r-j) - \sum_{V \subseteq [r-j-1]}t(V\cup \{r-j\}))\\
 a_{2,2^{r-j-1}+1} &= d - \nu(r+1)- \nu(r-j)+\sum_{V \subseteq [r-j-1]} t(V\cup \{ r-j\}) )
  \end{align*}
 where have used some ordering $V_i$ of the nonempty subsets of $[r-j-1]$ with $i \geq 2$.  Applying the result of Lemma \ref{rec} to expression \eqref{su claim j} yields $E(j+1)$. $E(r)$ is equal to 
\[
\frac{1}{\nu(r+1)!}\prod_{i=1}^r {d \choose \nu(i)}
\]
which, times $(d)_{\nu(r+1)}$, is equal to the right side of \eqref{su}. This completes the proof.

\end{proof} 

\begin{definition} 
For integers $m\geq n \geq 0$, define an $(m-n)$-multiset $\lambda$ 
\[
\lambda = (\lambda(n+1), \lambda(n+2), \ldots, \lambda(m))
\]
to be a multiset consiting of $m-n$ integers (counted with multiplicity) such that $1 \leq \lambda(i-1)\leq \lambda(i) \leq i$. Let $L(m,n)$ denote the set of such $(m-n)$-multisets. Let $\mathrm{mult}_i(\lambda)$ denote the number of times $i$ appears in $\lambda$.  Define the integer $c(\lambda)$ by 
\[
c_m(\lambda) = \prod_{i=1}^m {m-i+1  - \sum_{j=i+1}^m \mathrm{mult}_j(\lambda)\choose \mathrm{mult}_i(\lambda)}
\]

Define $L(m,m+1)$ to consist of the empty multiset $\emptyset$ with $c(\emptyset)=1$. Note that we may identify a multiset $\lambda \in L(m,n)$ with the same multiset in $L(m+1,n+1)$. 
\end{definition}

 \begin{lemma}\label{l s exp}
For an integer $m \geq 0$ and an indeterminate $d$,
  \begin{equation} \label{s exp}
  \prod_{i=1}^{m} (d-i+1+\sum_{j=1}^i s_j) = \sum_{i=0}^{m} (d)_{i} \sum_{\lambda \in L(m,i)} c_m(\lambda) \prod_{j=1}^{m} (s_j)_{\mathrm{mult}_j(\lambda)}.  
  \end{equation} 
   \end{lemma}
 \begin{proof}
  
We use induction on $m$. Equation \eqref{s exp} is true for $m=0$. Assume it is true for some $m \geq 0$. Let $A_{m,i}$ denote
\[
A_{m,i} = \sum_{\lambda \in L(m,i)} c_m(\lambda) \prod_{j=1}^{m} (s_j)_{\mathrm{mult}_j(\lambda)}.  
\] 
Then by the induction hypothesis we have 
\begin{equation} \label{s exp ind}
\prod_{i=1}^{m+1} (d-i+1+\sum_{j=1}^i s_j) = (\sum_{i=0}^m (d)_i A_{m,i}) (d-m+\sum_{j=1}^{m+1}s_j)
\end{equation}
It is sufficient to prove that the right side of the above equation is equal to 
\[
 (\sum_{i=0}^{m+1} (d)_i A_{m+1,i}).
\]
Using 
\[
(d)_i d = (d)_{i+1}+ i (d)_i
\]
we expand the right side of \eqref{s exp ind} and re-index to obtain 
\[
\sum_{i=0}^{m+1} (d)_i (A_{m,i-1}+ A_{m,i}(i-m+\sum_{j=1}^{m+1} s_j)).
\]
Therefore we must prove
\begin{equation} \label{s exp prove}
A_{m+1,i}=A_{m,i-1}+ A_{m,i}(i-m+\sum_{j=1}^{m+1} s_j).
\end{equation}

By definition, for $\lambda \in L(m,i)$, 
\[
\sum_{i=0}^m \mathrm{mult}_i(\lambda) = m-i.
\]
We thus have 
\begin{align*}
A_{m,i}(i-m+\sum_{j=1}^{m+1} s_j) &= \left (\sum_{\lambda \in L(m,i)} c_m(\lambda) \prod_{j=1}^{m} (s_j)_{\mathrm{mult}_j(\lambda)}\right)(i-m+\sum_{j=1}^{m+1} s_j)\\ 
&=  \sum_{\lambda \in L(m,i)} c_m(\lambda)( \prod_{j=1}^{m} (s_j)_{\mathrm{mult}_j(\lambda)})(\sum_{j=1}^{m+1}s_j - \mathrm{mult}_j(\lambda))\\ 
&= \sum_{\lambda \in L(m,i)} c_m(\lambda)( \sum_{k=1}^{m+1}\prod_{j=1}^{m+1} (s_j)_{\mathrm{mult}_j(\lambda)+\delta_{j,k}})
\end{align*}
where $\delta_{j,k}$ denotes the Kronecker delta. Next, given a fixed $\lambda \in L_{m+1,i}$, we compare the coefficients of  

\[
\prod_{j=1}^{m} 
(s_j)_{\mathrm{mult}_j(\lambda)}
\]
in \eqref{s exp prove}. The term $A_{m+1,i}$ contributes 
\[
c_{m+1}(\lambda),
\]
the term $A_{m,i-1}$ contributes 
\[
c_m(\lambda),
\]
and $A_{m,i}(i-m+\sum_{j=1}^{m+1} s_j)$ contributes 
\[
\sum_{i=1}^{m+1}c_m(\lambda - \{i\})
\]
where 
\[
\mathrm{mult}_i(\lambda - \{ i\})  = \mathrm{mult}_i(\lambda) -1
\]
and if 
\[
\mathrm{mult}_i(\lambda - \{ i\}) = -1
\]
then 
\[
c_m(\lambda - \{i\}) = 0. 
\]

Thus we must prove 
\begin{align} \label{s exp l}
\prod_{i=1}^{m+1} {m-i+2  - \sum_{j=i+1}^m \mathrm{mult}_j(\lambda)\choose \mathrm{mult}_i(\lambda)} &= \prod_{i=1}^{m} {m-i+1  - \sum_{j=i+1}^m \mathrm{mult}_j(\lambda)\choose \mathrm{mult}_i(\lambda)} \\ \nonumber
&+ \sum_{k=1}^{m} \prod_{i=1}^{m} {m-i+1  - \sum_{j=i+1}^m (\mathrm{mult}_j (\lambda)- \delta_{j,k})\choose \mathrm{mult}_i(\lambda) - \delta_{i,k})}.
\end{align}
For any integer $0\leq l \leq m$, we claim that the right side of \eqref{s exp l} is equal to $E(l)$, where 

and
\begin{align*}
E(l)  = F(l) G(l) + H(l)
\end{align*}
where 
\begin{align*}
F(l) &=  ( \prod_{i=1}^{l} {m-i+2  - \sum_{j=i+1}^m \mathrm{mult}_j(\lambda)\choose \mathrm{mult}_i(\lambda)})  \\ 
G(l) &= ( \prod_{i=l+1}^{m} {m-i+1  - \sum_{j=i+1}^m \mathrm{mult}_j(\lambda)\choose \mathrm{mult}_i(\lambda)}) \\ 
H(l) &= \sum_{k=l+1}^{m} \prod_{i=1}^{m} {m-i+1  - \sum_{j=i+1}^m (\mathrm{mult}_j (\lambda)- \delta_{j,k})\choose \mathrm{mult}_i(\lambda) - \delta_{i,k})}.
\end{align*}

We use induction on $l$. The claim is true for $l=0$. Assume it is true for some $l \geq 0$. Then 
\begin{align}
E(l)& =F(l) {m-l  - \sum_{j=l+2}^m \mathrm{mult}_j(\lambda)\choose \mathrm{mult}_{l+1}(\lambda)}   G(l+1)\label{term 1} \\ 
&+F(l) {m-l - \sum_{j=l+2}^m (\mathrm{mult}_j (\lambda))\choose \mathrm{mult}_{l+1}(\lambda) - 1)} G(l+1)\label{term 2}  \\ 
&+H(l+1)\nonumber
\end{align}

Adding the terms at lines \eqref{term 1} and \eqref{term 2}  using the identity 
\[
{x+1 \choose n} = {x \choose n}+ {x \choose n-1}
\]
yields $E(l+1)$. This completes the proof.
 \end{proof}
\begin{definition} 
Let $P(r,k)$ denote the set  of subsets of $[r]$ of size $k$. Let $x_i, 1 \leq i \leq r$ be a set of indeterminates. Given a subset $U \in P(r,k)$ with 
\[
U = (U(1), \ldots, U(k)),
\]
 let 
\[
x(U) = \prod_{i =1}^k (x_{U(i)}-i+1)
\]
\end{definition} 

\begin{lemma} \label{l s exp 2}
For an integer $r\geq 0$,
\begin{equation} \label{s exp 2}
\prod_{i=1}^r (y+x_i-i+1) = \sum_{i=0}^r (y)_i \sum_{U \in P(r,r-i)} x(U).
\end{equation}
\end{lemma}

\begin{proof}
 We use induction on $r$. Equation \eqref{s exp 2} is true for $r=0$. Assume it is true for some $r \geq 0$. Then 
 \[
 \prod_{i=1}^{r+1} (y+x_i-i+1)  = (\sum_{i=0}^r (y)_i \sum_{U \in P(r,r-i)} x(U)) (y+x_{r+1}-r)
 \]
 Now for $0 \leq i \leq r$
 \begin{align*}
&((y)_i \sum_{U \in P(r,r-i)} x(U)) (y - i+ x_{r+1}-(r-i))= (y)_{i+1} \sum_{U \in P(r,r-i)} x(U) + (y)_i \sum_{U \in P(r,r-i)} x(U) (x_{r+1}-(r-i))\\ 
 &= (y)_{i+1} \sum_{\substack{U \in P(r+1,r+1-(i+1))\\ r+1 \notin U}} x(U) + (y)_i \sum_{\substack{U \in P(r+1,r+1-i) \\ r+1 \in U}} x(U) 
 \end{align*}
 Thus the combined coefficient of $(y)_i$ is 
 \[
\sum_{U \in P(r+1,r+1-i) } x(U). 
 \]
 This completes the proof.
\end{proof}

\subsection{An encoding of simple directed graphs}

\begin{remark} \label{bijection}
We construct a bijection between DSG($d$) and a sequence of ``size-bounded" subsets of $[1,d]$ defined in Definition \ref{d sb}. This bijection yields an injection of the terms on the right side of \eqref{simple graph gen} into the left side. The remaining terms on the left side then cancel. 
\end{remark}

\begin{definition} \label{d sb}
For an integer $r \geq 1$, let $\sigma$ denote a sequence of subsets of $[r]$
\[
\sigma = (\sigma(1) \ldots \sigma(k)).
\]
We say that $\sigma$ is \emph{size-bounded} if $|\sigma(i)|\leq i$ for each $i$. Let $\mathcal{S}(k,r)$ denote the set of of such size-bounded sequences of length $k$. For an integer $j$, define the integer $l_j(\sigma)$ to be 
0 if $j \notin \sigma(i)$ for all $i$, and otherwise to be the smallest $i$ such that $j \in \sigma(i)$. For any set $U$ and element $x$, let $U - \{x\}$ denote the usual set subtraction, where 
\[
U - \{x\} = U
\] 
if $x \notin U$. 

Define the mapping $R_j$ 
\[
R_j : \mathcal{S}(k,r) \rightarrow \mathcal{S}(k-1,r)
\]
by 
\[ 
R_j(\sigma) = \begin{cases} 
(\sigma(1) \ldots, \sigma(k-1)) \text{ if $l_j(\sigma) = 0$ } \\
(\sigma(1)\ldots \sigma(l_j(\sigma) -1), \sigma(l_j(\sigma) +1) - \{ j\}, \ldots,  \sigma(k) - \{ j\}) \text{ otherwise }.
\end{cases}
\]
 Define the subset $\mathcal{B}(k,r) \subset \mathcal{S}(k,r)$ to consist of those $\sigma$ such that 
\[
R_i \circ R_{i+1} \circ \ldots \circ R_{r} (\sigma)
\] 
is size-bounded for each $i, 1 \leq i \leq r$. Define both $\mathcal{B}(0,1)$ and $\mathcal{S}(0,1)$ to consist of the empty sequence. 

Let $t$ denote a sequence of nonnegative integers 
\[
t = (t(1), \ldots, t(r)).
\]
Let $T(r)$ denote the set of such sequences of length $r$. Let $\rho_1$ be the mapping 
\[
\rho_1:  \mathcal{S}(k,r) \rightarrow T(r)
\] 
defined by 
\[
\rho_1(\sigma)(j) = \sum_{i=1}^k  \emph{\textbf{1}}(j \in \sigma(i))
\]
 \end{definition} 

\begin{definition} 
For an integer $r \geq 0$, let $A(r)$ denote the set of adjacency matrices of directed graphs on $r$ vertices that have no multiples edges and no 1- or 2-cycles. For example, if there is an edge from vertex $i$ to vertex $j$ in graph $G$, then the adjacency matrix $M$ of $G$ has $M_{i,j}=1$ and $M_{j,i}=0$. Let $\rho_2$ to be the mapping  
\[
\rho_2: A(r) \rightarrow \mathcal{T}(r)
\]
defined by 
\[
\rho_2(M)(j) = \sum_{i=1}^r M_{i,j}.
\]
\end{definition}

\begin{theorem} 
For an integer $r \geq 1$, there is a bijection $f$ 
\[
f:  A(r) \rightarrow \mathcal{B}(r-1,r) 
\]
such that 
\[
\rho_2(M) = \rho_1(f(M)).
\]
\end{theorem} 
\begin{proof}
We define $f$ inductively. For $r=1$, there is only one such adjacency matrix $M = [ 0 ]$, and $f(M)$ is the empty sequence. 

For an $r \geq 1$, assume $f$ is defined on $A(r)$. Let $M \in A(r+1)$. Let $M_0$ denote the $r\times r$ submatrix of $M$ obtained by taking the first $r$ rows and $r$ columns of $M$. By the induction hypothesis, $f(M_0)$ is defined; let $\sigma = f(M_0)$. Define the subset $U \in [r]$ by  
\[
U = \{j : M_{r+1,j}=1  \}.
\]

Suppose $M_{i,r+1}=0$ for all $i$. Then define $\sigma'$ by 
\[
 \sigma'(i) = 
 \begin{cases} 
 \sigma(i)  \text{ if $i \leq r-1$} \\
 U \text{ if $i = r$} 
 \end{cases}
 \]
 and define $f(M) = \sigma'$. 
 
Suppose there are $n\geq 1$ 1's in the $(r+1)$-column of $M$. Let $V$ be the sequence 
\[
v = ( v(1), \dots, v(r -|U| ))
\]
such that $v(i) =M_{g(i), r+1}$ where $g(i)$ is the unique increasing function from $[r-|U|]$ to $[r] - \{ U\}$. 
Let $i_1$ be the smallest $i$ such that $v(i)=1$. 
For $1 \leq i \leq r$ define $\sigma'$ by 
\[
\sigma'(i) =\begin{cases} 
\sigma(i) \text{ if $i < |U|+i_1$ } \\ 
U \cup \{r+1\} \text{ if $i = |U|+i_1$ } \\
\sigma(i-1) \text{ if $i > |U|+i_1$ and $v(i-|U|)=0$ } \\ 
\sigma(i-1) \cup \{ r+1\} \text{ if $i > i_1$ and $v(i-|U|)=1$ }
\end{cases} 
\]
 Then define $f(M) = \sigma'$. It is straightforward to check that $\rho_1(\sigma') = \rho_2(M)$. By construction  $R_{r+1}(\sigma') = \sigma$, and by the induction hypothesis  $\sigma \in \mathcal{B}(r-1,r)$, so $\sigma' \in \mathcal{B}(r,r+1) $. 
 
We prove that $f$ is a bijection also by induction. It is a bijection for $r=1$. Assume it is for some $r \geq 1$. Given $\sigma \in \mathcal{B}(r,r+1)$ we can construct the bottom row and rightmost column of the adjacent matrix as follows. If there is no $i$ such that $r+1 \in \sigma(i)$, then the bottom row is determined by $\sigma(r)$, and the rightmost column consists of all 0's. Otherwise the bottom row is determined by $\sigma(i_1)$, where $i_1$ is the smallest $i$ such that $r+1 \in \sigma(i)$, and the rightmost column is determined by the indices $i$ of those $\sigma(i)$ that contain the remaining $(r+1)$'s. The upper left $r \times r$ submatrix is then determined inductively. This completes the proof.
\end{proof}

\section{Further Work} 
 \begin{enumerate} [label = $\bullet$]
 
\item Extend the bijection of Remark \ref{bijection} to the remaining terms. 

\item Find analogues of the factored formulas of \cite{DeFranco 2} for formal zeros of multi variable polynomials. 

\item Find NRS-like algorithms for the formal zeros of \cite{DeFranco 2} faith integer $\beta > 1$. 

\end{enumerate}

\end{document}